\documentclass[reqno,11pt]{amsart}
\usepackage[letterpaper, margin=1in]{geometry}
\RequirePackage{amsmath,amssymb,amsthm,graphicx,mathrsfs,url,slashed,subcaption}
\RequirePackage[usenames,dvipsnames]{xcolor}
\RequirePackage[colorlinks=true,linkcolor=Red,citecolor=Green]{hyperref}
\RequirePackage{amsxtra}
\usepackage{cancel}
\usepackage{tikz-cd}

\makeatletter
\@namedef{subjclassname@2020}{\textup{2020} Mathematics Subject Classification}
\makeatother

\title{Minimal laminations and level sets of $1$-harmonic functions}
\author{Aidan Backus}
\address{Department of Mathematics, Brown University}
\email{aidan\_backus@brown.edu}
\date{\today}

\newcommand{\NN}{\mathbf{N}}
\newcommand{\ZZ}{\mathbf{Z}}

\newcommand{\RR}{\mathbf{R}}

\newcommand{\Hyp}{\mathbf H}
\newcommand{\Sph}{\mathbf S}

\newcommand{\Ball}{\mathbf{B}}

\newcommand*\dif{\mathop{}\!\mathrm{d}}

\DeclareMathOperator{\dist}{dist}

\DeclareMathOperator{\supp}{supp}
\DeclareMathOperator{\tr}{tr}

\newcommand{\Leaves}{\mathscr L}

\newcommand{\Hypspace}{\mathscr H}

\newcommand{\Two}{\mathrm{I\!I}}

\newcommand{\normal}{\mathbf n}

\newcommand{\Lip}{\mathrm{Lip}}
\newcommand{\Riem}{\mathrm{Riem}}
\newcommand{\Ric}{\mathrm{Ric}}

\newcommand{\dfn}[1]{\emph{#1}\index{#1}}

\newcommand{\loc}{\mathrm{loc}}
\newcommand{\cpt}{\mathrm{cpt}}

\newtheorem{theorem}{Theorem}[section]

\newtheorem{lemma}[theorem]{Lemma}

\newtheorem{proposition}[theorem]{Proposition}
\newtheorem{corollary}[theorem]{Corollary}
\newtheorem{conjecture}[theorem]{Conjecture}

\newtheorem{mainthm}{Theorem}

\theoremstyle{definition}
\newtheorem{definition}[theorem]{Definition}

\newtheorem{example}[theorem]{Example}



\numberwithin{equation}{section}

\def\XXint#1#2#3{{\setbox0=\hbox{$#1{#2#3}{\int}$ }
\vcenter{\hbox{$#2#3$ }}\kern-.6\wd0}}

\usepackage[backend=bibtex,style=alphabetic,giveninits=true]{biblatex}

\renewbibmacro{in:}{}
\DeclareFieldFormat{pages}{#1}

\begin{document}
\begin{abstract}
We collect several results concerning regularity of minimal laminations, and governing the various modes of convergence for sequences of minimal laminations.
We then apply this theory to prove that a function has locally least gradient (is $1$-harmonic) iff its level sets are a minimal lamination; this resolves an open problem of Daskalopoulos and Uhlenbeck.
\end{abstract}

\maketitle



\section{Introduction}
The space of codimension-$1$ laminations by minimal submanifolds on a Riemannian manifold has been topologized in several different ways.
Thurston \cite[Chapter 8]{thurston1979geometry} introduced both his geometric topology as well as the vague topology on the space of measured geodesic laminations.
Independently of Thurston, Colding and Minicozzi \cite[Appendix B]{ColdingMinicozziIV} introduced a topology that emphasized not the laminations themselves, but rather the coordinate charts which flatten them.

We establish a regularity theorem for minimal laminations, which implies compactness properties for the aforementioned topologies.
We also show that a current is Ruelle-Sullivan with respect to a minimal lamination if and only if it is locally the exterior derivative of a function of least gradient, generalizing a theorem of Daskalopoulos and Uhlenbeck \cite[Theorem 6.1]{daskalopoulos2020transverse} and strengthing an unpublished result of Auer and Bangert \cite{Auer01, Auer12}.

\subsection{Minimal laminations}\label{Lams sections}
Throughout this paper, we fix an interval $I \subset \RR$, a box $J \subset \RR^{d - 1}$, and a smooth Riemannian manifold $M = (M, g)$ of dimension $d \geq 2$.
We do not take $M$ to have a boundary, so that if $E \subseteq M$ is a compact set, $E$ cannot contain a singularity of $M$.

\begin{definition}
A (codimension-$1$) \dfn{laminar flow box} is a $C^0$ coordinate chart $F: I \times J \to M$ and a compact set $K \subseteq I$, such that for each $k \in K$, $F|_{\{k\} \times J}$ is a $C^1$ embedding, and the \dfn{leaf} $F(\{k\} \times J)$ is a $C^1$ complete hypersurface in $F(I \times J)$.
Two laminar flow boxes belong to the same \dfn{laminar atlas} if the transition maps between them preserve the germs of the local leaf structure.
\end{definition}

\begin{definition}
A \dfn{lamination} $\lambda$ consists of a nonempty closed set $S \subseteq M$, called its \dfn{support}, and a maximal laminar atlas $\{(F_\alpha, K_\alpha): \alpha \in A\}$ such that in the image $U_\alpha$ of each flow box $F_\alpha$,
$$S \cap U_\alpha = F_\alpha(K_\alpha \times J).$$
If $\lambda$ is a lamination in the image of a flow box $F$, and $N := F(\{k\} \times J)$ is a leaf of $\lambda$, we call $k$ the \dfn{label} of $N$.
A \dfn{foliation} is a lamination with support $S = M$.
\end{definition}

Summarizing the above definitions, a lamination is a nonempty closed set $S$ with a $C^0$ local product structure which locally realizes it as $K \times J$ for some compact set $K \subset \RR$.
In some sources, including \cite{Auer01}, laminations are not required to have a $C^0$ local product structure, but are only required to have disjoint leaves.

\begin{definition}
We call a lamination $C^r$ (resp. \dfn{Lipschitz}) if its flow boxes are $C^r$ (resp. Lipschitz) coordinate charts, and say that it is \dfn{tangentially $C^r$} if for each flow box $(F, K)$, $F|_{\{k\} \times J}$ is a $C^r$ embedding for $k \in K$.\footnote{Such laminations are also known as $C^r$ \dfn{along leaves} \cite{Morgan88}.}
\end{definition}

In particular, we assume that laminations are $C^0$ and tangentially $C^1$; the latter assertion implies that the flow box can push forward the normal vector to each leaf, and in particular that the mean curvature to each leaf is well-defined in a distributional sense.
The Lipschitz regularity is particularly natural in light of the previous results of \cite{Solomon86, Zeghib04}.

In this paper we shall focus on laminations with minimal leaves\footnote{The word ``minimal'' is overloaded. In \cite{daskalopoulos2020transverse}, a \dfn{minimal lamination} is a lamination $\lambda$ in which every leaf is dense in $\supp \lambda$.
We adopt the terminology of \cite{ColdingMinicozziIV}.} and transverse measures.

\begin{definition}
A lamination $\lambda$ is \dfn{minimal} if its leaves $F_\alpha(\{k\} \times J)$ have zero mean curvature, and is \dfn{geodesic} if, in addition, $d = 2$.
\end{definition}

\begin{definition}\label{transverse measure definition}
Let $\lambda$ be a lamination with atlas $A$.
A \dfn{transverse measure} to $\lambda$ consists of Radon measures $\mu_\alpha$ with $\supp \mu_\alpha = K_\alpha$, $\alpha \in A$, such that each transition map $\psi_{\alpha \beta}$ is measure-preserving:
$$\mu_\alpha|_{K_\alpha \cap K_\beta} = \psi_{\alpha \beta}^* (\mu_\beta|_{K_\alpha \cap K_\beta}).$$
The pair $(\lambda, \mu)$ is called a \dfn{measured lamination}.
\end{definition}

We assume that every transverse measure has full support, $\supp \mu_\alpha = K_\alpha$.
Therefore not every lamination $\lambda$ admits a transverse measure; for example, this happens if $\lambda$ has an isolated leaf which meets a compact transverse curve on a countable but infinite set \cite[Theorem 3.2]{Morgan88}.

\begin{definition}
Let $(\lambda, \mu)$ be a measured oriented lamination, with atlas $A$ and a subordinate partition of unity $(\chi_\alpha)$.
The \dfn{Ruelle-Sullivan current} $T_\mu$ associated to $(\lambda, \mu)$ is defined for all compactly supported $(d - 1)$-forms $\varphi$ by
\begin{equation}\label{RS current}
\int_M T_\mu \wedge \varphi := \sum_{\alpha \in A} \int_{K_\alpha} \left[\int_{\{k\} \times J} (F_\alpha^{-1})^* (\chi_\alpha \varphi) \right] \dif \mu_\alpha(k).
\end{equation}
\end{definition}

See Appendix \ref{portmanteau appendix} for generalities on currents.
The Ruelle-Sullivan current was introduced by \cite{Ruelle75}, and we review its properties in \S\ref{RS prelims}.
In particular we show that $T_\mu$ makes sense (as a distributional section of a suitable line bundle) even if $\lambda$ is not orientable.

\subsection{Regularity of minimal laminations}
The definitions of \S\ref{Lams sections} are tedious to work with, both because one has to prove the existence of flow boxes which flatten sets which may be extremely rough, and because one has no quantitative control on said flow boxes.
However, if we have curvature bounds on the leaves and on the underlying manifold $M$, our first main theorem drastically changes the story: it shows that the lamination $\lambda$ can be reconstructed from its set of leaves, in such a way that the flow boxes for $\lambda$ are under control in the Lipschitz and tangentially $C^\infty$ sense.
Here, \dfn{tangential $C^\infty$} is the topology defined by seminorms $f \mapsto \|\nabla_N^m f\|_{C^0}$, where $N$ ranges over leaves of the given lamination $\lambda$, $\nabla_N$ is the Levi-Civita connection on $N$, $m$ ranges over $\NN$ (including $0$); thus we ignore derivatives normal to $N$.

\begin{mainthm}\label{regularity theorem}
Let $K := \|\Riem_M\|_{C^0}$ and let $i$ be the injectivity radius of $M$, and suppose that $K < \infty$, $i > 0$.
Let $\mathcal S$ be a nonempty set of disjoint minimal hypersurfaces in $M$, such that for every $N \in \mathcal S$,
\begin{equation}\label{curvature bound in regularity}
	\|\Two_N\|_{C^0} \leq A,
\end{equation}
and that $\bigcup_{N \in \mathcal S} N$ is a closed subset of $M$. Then:
\begin{enumerate}
\item There exists a Lipschitz minimal lamination $\lambda$ whose leaves are exactly the elements of $\mathcal S$.
\item There exists a Lipschitz line bundle on $M$ which is normal to every leaf of $\lambda$.
\item There exist constants $L = L(A, K, i) > 0$ and $r = r(A, K, i) > 0$, and a Lipschitz laminar atlas $(F_\alpha)$ for $\lambda$, such that for every $\alpha$,
\begin{equation}\label{conorm of flow box}
	\max(\Lip(F_\alpha), \Lip(F_\alpha^{-1})) \leq L,
\end{equation}
and the image of $F_\alpha$ contains a ball of radius $r$.
\item $F_\alpha$ and $F_\alpha^{-1}$ are tangentially $C^\infty$, with seminorms only depending on $A, K, i$.
\end{enumerate}
\end{mainthm}

In the remainder of this paper we prove two consequences of Theorem \ref{regularity theorem}: a characterization of minimal laminations (Theorem \ref{main thm}) and a compactness theorem (Theorem \ref{compactness theorem}), which we state below.
In both theorems, a bound on the curvature will be necessary in order to invoke Theorem \ref{regularity theorem}.

\begin{definition}
A sequence $(\lambda_n)$ of laminations has \dfn{bounded curvature} if there exists $C > 0$ such that for any $n$ and any leaf $N$ of $\lambda_n$, the second fundamental form satisfies $\|\Two_N\|_{C^0} \leq C$.
\end{definition}

Several similar results to Theorem \ref{regularity theorem} have appeared in the literature already, but Theorem \ref{regularity theorem} strengthens and clarifies them.
To our knowledge, the first related result is due to Solomon \cite[Theorem 1.1]{Solomon86}, which we improve on in several ways:
\begin{enumerate}
\item \label{foliation to lamination} Solomon's proof is for minimal foliations in $\RR^d$.
\item We obtain estimates which only depend on the curvatures of the leaves and $M$, and on the injectivity radius $i$; they do not depend on the regularity of a given $C^0$ laminar atlas.
\item In fact, we do not even assume the existence of a $C^0$ laminar atlas.
\end{enumerate}
As Solomon notes, it is easy to extend his proof to minimal foliations in a Riemannian manifold $M$; the key point of (\ref{foliation to lamination}) is that we would like Theorem \ref{regularity theorem} to be true for minimal \emph{laminations}.

Our work is closest to a compactness theorem due to Colding and Minicozzi for minimal laminations of a Riemannian manifold \cite[Appendix B]{ColdingMinicozziIV}.
Their new idea is to fill in the gaps between the leaves in Solomon's constructions by linear interpolation.
However, Colding and Minicozzi assume that the laminations have finitely many leaves, and that the curvature bound (\ref{curvature bound in regularity}) implies that all of the leaves can be represented as graphs at once.
Indeed, \emph {a priori}, the leaves could fail to be close to parallel, and then it would not be possible to construct a coordinate chart in which they are all graphs.

We eliminate such assumptions by showing that members of $\mathcal S$ must be ``close to parallel on small scales'', where the scale is governed by $A, K$.
Otherwise, since the scale is small, we may replace the elements of $\mathcal S$ by their tangent spaces, which would then intersect, contradicting the disjointness of $\mathcal S$.
This approach was already suggested by Thurston \cite[\S8.5]{thurston1979geometry} in the case of geodesic laminations, though he omitted the details. 

Using completely different techniques, Daskalopoulos and Uhlenbeck \cite[Proposition 7.3]{daskalopoulos2020transverse} obtained a version of Theorem \ref{regularity theorem} without any $C^0$ dependence, under the assumption that $M$ is a closed hyperbolic surface.
The key point of their argument is that the exponential map sends lines to geodesics, so it provides a much shorter proof of Theorem \ref{regularity theorem}, at the price of only working in dimension $2$.

\subsection{Applications to \texorpdfstring{$1$-harmonic}{one-harmonic} functions}\label{FLG section}
The main result of this paper realizes the Ruelle-Sullivan current of a minimal lamination as the exterior derivative of a function of locally least gradient, and vice versa.

Write $BV_\loc(M), BV(M), BV_\cpt(M)$ for the spaces of functions of locally bounded variation, bounded variation, and bounded variation and compact support, respectively.
We review these function spaces in Appendix \ref{BV appendix}.

\begin{definition}
Let $u \in BV_\loc(M)$.
\begin{enumerate}
\item $u$ has \dfn{least gradient} in $M$ if for every $v \in BV_\cpt(M)$,
$$\int_{\supp v} \star |\dif u| \leq \int_{\supp v} \star |\dif (u + v)|.$$
\item $u$ has \dfn{locally least gradient} if there exists a cover $\mathcal U$ of $M$ by open sets with smooth boundary such that for any $U \in \mathcal U$, $u|_U$ has least gradient.
\end{enumerate}
\end{definition}

The notion of function of least gradient essentially goes back to work of Miranda \cite{Miranda66,Miranda67} and Bombieri, de Giorgi, and Giusti \cite{BOMBIERI1969} on area-minimizing hypersurfaces.
We refer to the monograph of G\'orny and Maz\'on \cite{gorny2024leastgradient} for a detailed treatment.

Functions of (locally) least gradient are weak solutions, in a suitable sense, of the \dfn{$1$-Laplace equation}
\begin{equation}\label{1Laplacian}
	\nabla \cdot \left(\frac{\nabla u}{|\nabla u|}\right) = 0
\end{equation}
and so we also call functions of locally least gradient \dfn{$1$-harmonic functions} \cite{Mazon14}.
Observe that formally, (\ref{1Laplacian}) implies that the level sets $\partial \{u > y\}$ of $u$ are minimal hypersurfaces.
See Appendix \ref{boundary conventions} for our conventions on boundaries of measurable sets.

The above definitions are slightly nonstandard, and this is necessary both because the notion of lamination is local (in the sense that it is decided by restrictions to an arbitrarily fine open cover) and because in typical applications one wants to allow the domain $M$ to be unbounded.
In \S\ref{least gradient formulation} we discuss these subtleties.

\begin{mainthm}\label{main thm}
Suppose that $2 \leq d \leq 7$.
\begin{enumerate}
\item Let $u$ be a function of locally least gradient on $M$ which is not constant.
Then:
\begin{enumerate}
\item $\bigcup_{y \in \RR} \partial \{u > y\} \cup \partial \{u < y\}$ is the support of a Lipschitz minimal lamination $\lambda$.
\item The leaves of $\lambda$ are exactly the connected components of $\partial \{u > y\}$ or $\partial \{u < y\}$, with $y$ ranging over $\RR$.
\item There exists a measured oriented structure on $\lambda$ whose Ruelle-Sullivan current is $\dif u$.
\item If $u$ has least gradient, then the leaves of $\lambda$ are area-minimizing.
\end{enumerate}
\item Conversely, if $H^1(M, \RR) = 0$ and $\lambda$ is a measured oriented minimal lamination, then:
\begin{enumerate}
\item If $\lambda$ has bounded curvature, then any primitive $u$ of the Ruelle-Sullivan current of $\lambda$ has locally least gradient.
\item If the leaves of $\lambda$ are area-minimizing, then $u$ has least gradient.
\end{enumerate}
\end{enumerate}
\end{mainthm}

In general, we must use both sublevel sets and superlevel sets in the statement of Theorem \ref{main thm}.
For example, the function
$$u(x, y) := 1_{\{x \leq 0\}} x$$
has least gradient on $\RR^2$.
Then $\{x = 0\}$ is a leaf of the lamination and bounds $\{u < 0\}$ but does not bound a superlevel set.
However, the set of leaves arising from sublevel sets but not superlevel sets is countable, and we can do away with sublevel sets entirely if we assume that $\dif u$ has full support, by an argument similar to \cite[Lemma 2.11]{górny2018}.

Even if we wished to allow for singular minimal hypersurfaces, we would not be able to establish a lamination if $d \geq 8$; in fact, if $x, y \in \RR^4$, then the function
$$u(x, y) := 1_{\{|x|^2 > |y|^2\}}$$
associated to the Simons cone has least gradient on $\RR^8$ \cite[Theorem A]{BOMBIERI1969}, but there is no way to establish flow box coordinates near $0$.
Since we use estimates on stable minimal hypersurfaces which depend strongly on the dimension, the following natural conjecture does not seem provable using the methods of this paper.

\begin{conjecture}
Suppose that $d \geq 8$. Then every function of locally least gradient $u: M \to \RR$ has a singular set $S_u$ of codimension $8$, such that on $M \setminus S_u$, the analogue of Theorem \ref{main thm} holds.
\end{conjecture}

The main ingredients in the proof of Theorem \ref{main thm} are Theorem \ref{regularity theorem}, the regularity theory of minimal hypersurfaces, and curvature estimates on stable minimal hypersurfaces due to Schoen, Simon, and Yau \cite{Schoen75,Schoen81}.
With these ingredients in place, it remains to show that the stability radii of the level sets of a function of locally least gradient are bounded from below, and locally the area of the level sets is bounded from above; this gives uniform curvature estimates on the level sets.

A similar result to Theorem \ref{main thm}, proven with somewhat different methods, was announced but never published by Auer and Bangert \cite{Auer01, Auer12}, who claimed to establish that a locally minimal $1$-current is Ruelle-Sullivan for a lamination in a weaker sense than ours.
In particular, it does not seem that one can extract Lipschitz regularity directly from their methods.

Our motivation for Theorem \ref{main thm} is to generalize the work of Daskalopoulos and Uhlenbeck on $\infty$-harmonic maps from a closed hyperbolic surface to $\Sph^1$ \cite{daskalopoulos2020transverse}, which associates to each such map a geodesic lamination $\lambda$ and function $v$ of locally least gradient on the universal cover such that $\dif v$ drops to a Ruelle-Sullivan current for a sublamination of $\lambda$.
Inspired by this theorem, Daskalopoulos and Uhlenbeck conjectured that for any function of locally least gradient on the hyperbolic plane $\Hyp^2$, $\dif u$ should be Ruelle-Sullivan for some (possibly not maximum-stretch) geodesic lamination \cite[Problem 9.4]{daskalopoulos2020transverse}, and conversely that if $T$ is a Ruelle-Sullivan current for some geodesic lamination, then local primitives of $T$ have locally least gradient \cite[Conjecture 9.5]{daskalopoulos2020transverse}.
Of course such results are special cases of Theorem \ref{main thm}.
We shall revisit the connection between Theorem \ref{main thm} and the $\infty$-Laplacian in \cite{BackusInfinityMaxwell1}, where we explain how one can view the $1$-Laplacian as the convex dual problem to the problem of constructing a calibration of a minimal lamination, which is given by a system of ``$\infty$-elliptic'' equations.

We stress that Theorem \ref{main thm} is an interior result.
We allow $M$ to have a boundary, infinite ends, or punctures, but do not study the limiting behavior of the lamination near those points.
The behavior of functions of least gradient near an infinite end is heavily constrained by the global behavior of area-minimizing hypersurfaces \cite[\S4.4]{górny2021}, which is outside the scope of this paper.

In \S\ref{1harmonic apps} we use Theorem \ref{main thm} prove a generalization of G\'orny's decomposition of functions of least gradient \cite[Theorem 1.2]{górny2017planar} to our setting.
A simplified version of the statement is as follows:

\begin{corollary}
Let $u$ be a function of locally least gradient, and $d \leq 7$.
Then we can locally write $u$ as the sum of an absolutely continuous function of least gradient, a Cantor function of least gradient, and a jump function of least gradient.
\end{corollary}


\subsection{Spaces of minimal laminations}\label{LamSpace section}
In the literature, there are at least three different topologies on the space of laminations on a Riemannian manifold $M$, which we now recall.

Thurston's geometric topology \cite[Chapter 8]{thurston1979geometry} says that a lamination $\lambda'$ is close to a lamination $\lambda$ if every leaf of $\lambda$ is close to a leaf of $\lambda'$ at least locally, and the same holds for their normal vectors $\normal$.

\begin{definition}
We define the basic open sets in \dfn{Thurston's geometric topology} to be defined by a lamination $\lambda$, $x \in \supp \lambda$, and $\varepsilon > 0$: the basic open set $\mathscr N(\lambda, x, \varepsilon)$ is the set of all laminations $\kappa$ such that there exists $y \in \supp \kappa \cap B(x, \varepsilon)$ such that the normal vectors are close: $\dist(\normal_\lambda(x), \normal_\kappa(y)) < \varepsilon$.
\end{definition}

A sequence of laminations $(\lambda_i)$ converges to a lamination $\lambda$ in Thurston's geometric topology iff, for every leaf $N$ of $\lambda$, every $x \in N$, and every $\varepsilon > 0$, there exists $i_{\varepsilon, x} \in \NN$ such that for every $i \geq i_{\varepsilon, x}$, $\supp \lambda_i$ intersects $B(x, \varepsilon)$, and for some $x_i \in B(x, \varepsilon) \cap \supp \lambda_i$,
$$\dist_{SM}(\normal_{\lambda_i}(x_i), \normal_\lambda(x)) < 2\varepsilon.$$
It is straightforward to show that Thurston's geometric topology does not depend on the choice of Riemannian metric on $M$, or the choice of extension of the distance function on $M$ to its sphere bundle $SM$, which are implicit in the statement thereof.
However, the limiting lamination is not unique, as if $\lambda_i \to \lambda$ and $\lambda'$ is a sublamination of $\lambda$, then $\lambda_i \to \lambda'$.
In particular, Thurston's topology is not Hausdorff, and we say that $\lambda$ is a \dfn{maximal limit} of a sequence $(\lambda_i)$ if $\lambda_i \to \lambda$ and for every $\lambda'$ such that $\lambda_i \to \lambda'$, $\lambda'$ is a sublamination of $\lambda$.

Independently of Thurston, Colding and Minicozzi \cite[Appendix B]{ColdingMinicozziIV} defined a sequence of laminations to converge ``if the corresponding coordinate maps converge;'' that is, if the laminar atlases converge.
This of course says nothing about the limiting set of leaves and in the sequel paper \cite{ColdingMinicozziV} they additionally impose that the sets of leaves converge ``as sets.''

In this paper we consider a similar condition to the one in \cite{ColdingMinicozziV}, which we believe to be more natural: that the laminar atlases converge and that the laminations themselves converge in Thurston's geometric topology.
To be more precise:

\begin{definition}
A sequence $(\lambda_i)$ of laminations \dfn{flow-box converges} in a function space $X$ to $\lambda$ if it converges in Thurston's geometric topology, and there exists a laminar atlas $(F_\alpha)$ for $\lambda$ such that for each $\alpha$, $F_\alpha$ and $(F_\alpha)^{-1}$ are limits in $X$ of flow boxes $F_\alpha^i$, $(F_\alpha^i)^{-1}$ in laminar atlases for $\lambda_i$.
\end{definition}

Let $C^{1-}$ denote the topology generated by the H\"older norms $\|u\|_{C^\theta}$, $\theta \in [0, 1)$.
Thus $u_n \to u$ in $C^{1-}$ iff for every $\theta \in [0, 1)$, $\|u_n - u\|_{C^\theta} \to 0$.
We shall mainly be interested in flow-box convergence in the $C^{1-}$ and tangential $C^\infty$ senses.

Next we recall convergence of laminations equipped with transverse measures.
This definition is equivalent to the definition of Thurston for measured geodesic laminations in $\Hyp^2$, see \S\ref{hyperbolic equivalence}.
In \S\ref{RS prelims} we define the Ruelle-Sullivan current of a possibly nonorientable measured lamination, and in Appendix \ref{portmanteau appendix} we review the vague topology.

\begin{definition}
A sequence of measured laminations $(\lambda_i, \mu_i)$ \dfn{converges} to $(\lambda, \mu)$ if their Ruelle-Sullivan currents vaguely converge.
\end{definition}

Filling in some of the details of the argument of Colding and Minicozzi \cite[Appendix B]{ColdingMinicozziIV}, it follows from the regularity theorem, Theorem \ref{regularity theorem}, that once we have a bound on the curvatures of the leaves, every sequence of laminations has convergent subsequences in each of the above modes of convergence.

\begin{mainthm}\label{compactness theorem}
Let $(\lambda_n)$ be a sequence of minimal laminations.
Assume that $(\lambda_n)$ has bounded curvature, and there exists a compact set $E \subseteq M$ such that for every $n$ and every leaf $N$ of $\lambda_n$, $N \cap E$ is nonempty. Then there is a minimal lamination $\lambda$ such that:
\begin{enumerate}
\item A subsequence of $(\lambda_n)$ converges in the $C^{1-}$ and tangentially $C^\infty$ flow box topology to $\lambda$.
\item $\lambda$ is a maximal limit of $(\lambda_n)$ in Thurston's geometric topology.
\item Let $(\mu_n)$ be a sequence of measures, with $\mu_n$ transverse to $\lambda_n$. Assume that $(T_{\mu_n})$ is vaguely bounded, and for some $\varepsilon > 0$, $\int_E \star |T_{\mu_n}| \geq \varepsilon$. Then there exists a sublamination $\lambda'$ and a transverse measure $\mu$ to $\lambda'$ such that $(\lambda_n, \mu_n) \to (\lambda', \mu)$.
\end{enumerate}
\end{mainthm}

To avoid trivialities (such as every sequence of laminations Thurston-converging to the empty lamination) we have taken as part of the definition that a lamination is nonempty.
Thus it is necessary to include a condition which prevents the laminations $\lambda_n$ from escaping to infinity (or, if $M$ is a bounded domain, accumulating on the boundary) and we accomplish this by requiring that every leaf of the $\lambda_n$s meet a compact set.
This technicality can be avoided when $M$ is a closed manifold, and the below corollary immediately follows:

\begin{corollary}
Assume that $M$ is a closed manifold.
Let $(\lambda_n)$ be a sequence of minimal laminations of bounded curvature on $M$.
Then there exists a minimal lamination $\lambda$ such that:
\begin{enumerate}
\item A subsequence of $(\lambda_n)$ converges in the $C^{1-}$ and tangentially $C^\infty$ flow box topology to $\lambda$.
\item $\lambda$ is a maximal limit of $(\lambda_n)$ in Thurston's geometric topology.
\item Let $(\mu_n)$ be a sequence of measures, with $\mu_n$ transverse to $\lambda_n$. Assume that for some $\varepsilon > 0$, $\varepsilon \leq \int_M \star |T_{\mu_n}| \leq 1/\varepsilon$. Then there exists a sublamination $\lambda'$ of $\lambda$, and a transverse measure $\mu$ to $\lambda'$, such that $(\lambda_n, \mu_n) \to (\lambda', \mu)$.
\end{enumerate}
\end{corollary}

Since not every lamination admits a transverse measure, it is natural to allow $\lambda'$ to be a sublamination of $\lambda$ in Theorem \ref{compactness theorem}.
For example, there is a hyperbolic punctured torus $M$ and a geodesic lamination $\lambda$ on $M$ which is the Thurston limit of finite sums of closed geodesics, but has two leaves: a closed geodesic $\lambda'$, and a geodesic which accumulates on itself countably many times \cite[Figure 1(b)]{Dumas20}.
Thus any limiting measure must concentrate on $\lambda'$.

In \S\ref{relationships between modes}, we use Theorem \ref{compactness theorem} to explain how the above modes of convergence are related:

\begin{corollary}
Let $(\lambda_n, \mu_n)$ be a sequence of measured minimal laminations of bounded curvature and $(\lambda, \mu)$ a measured minimal lamination.
If $d \leq 7$ and $(\lambda_n, \mu_n) \to (\lambda, \mu)$, then $\lambda_n \to \lambda$ in the $C^{1-}$ and tangential $C^\infty$ flow box topologies, hence in Thurston's geometric topology.
\end{corollary}


\subsection{Notation and conventions}
The operator $\star$ is the Hodge star on $M$, thus $\star 1$ is the Riemannian measure of $M$.
We denote the musical isomorphisms by $\sharp, \flat$, and Sobolev spaces by $W^{s, p}$.
The manifold $\Ball^d$ is the unit ball in $\RR^d$, $\Sph^d$ is the unit sphere in $\RR^{d + 1}$, and $\Hyp^d$ is the hyperbolic space.

The $\delta$-dimensional Hausdorff measure is $\mathcal H^\delta$, normalized so that if $\delta$ is an integer, then $\mathcal H^\delta$ is $\delta$-dimensional Riemannian measure.
We write $\omega_\delta$ for $\mathcal H^\delta(\Ball^\delta)$.

If we specify that $M$ is a manifold with boundary, we specifically mean that $M$ has a smooth boundary.
By \dfn{injectivity radius} we mean the minimum of the injectivity radius and distance to $\partial M$.
If $E$ is a set of locally finite perimeter, we assume that the boundaries of $E$ in the sense of point-set topology and measure theory agree; see Appendix \ref{boundary conventions}.

By a \dfn{hypersurface} we mean a $C^1$ submanifold of codimension $1$.
We write $\normal_N$ for the normal vector (or conormal $1$-form) for a hypersurface $N$, $\nabla_N$ for the Levi-Civita connection, and $\Two_N := \nabla_N \normal_N$ for the second fundamental form if it is defined.

For a map $F: X \to Y$ between metric spaces, we write $\Lip(F)$ for its Lipschitz constant.
If $X, Y$ are connected Riemannian manifolds, one of which is $1$-dimensional, then we have $\Lip(F) = \|\dif F\|_{L^\infty}$.

We write $A \lesssim_\theta B$ or $A = O_\theta(B)$ to mean that for some constant $C \geq 1$, which only depends on $\theta$, $A \leq CB$.

\subsection{Outline of the paper}
The rest of the paper is organized as follows:
\begin{itemize}
\item In \S\ref{Regularity}, we prove the regularity theorem, Theorem \ref{regularity theorem}.
\item In \S\ref{Prelims}, we develop basic facts about Ruelle-Sullivan currents, and $1$-harmonic functions, that we shall use throughout the remainder of the paper. This section is independent of \S\ref{Regularity}, but depends on Appendix \ref{boundary appendix}.
\item In \S\ref{1harmonic sec}, we prove the equivalence of $1$-harmonic functions and measured oriented minimal laminations, Theorem \ref{main thm}, and apply it to study $1$-harmonic functions. This section relies on \S\ref{Regularity}, \S\ref{Prelims}, and Appendices \ref{boundary appendix} and \ref{locally minimizing appendix}.
\item In \S\ref{CompactnessSec}, we prove the compactness theorem, Theorem \ref{compactness theorem}, and explore the consequences for how the different modes of convergence are related to each other. This section applies \S\ref{Regularity}, \S\ref{Prelims}, and Appendix \ref{boundary appendix} for the proof of Theorem \ref{compactness theorem}, but the consequences of it also apply \S\ref{1harmonic sec}.
\item In Appendix \ref{boundary appendix}, we recall various technical results of geometric measure theory that we shall need.
\item In Appendix \ref{locally minimizing appendix}, we give a short proof that the radius of a ball in which a minimal hypersurface is area-minimizing is controlled from below by the curvature. The proof applies both \S\ref{Regularity} and \S\ref{Prelims}.
\end{itemize}


\subsection{Acknowledgements}
I would like to thank Georgios Daskalopoulos for suggesting this project and for many helpful discussions; Victor Bangert for allowing me to read the draft \cite{Auer12}; Christine Breiner, Wojciech G\'orny, and the anonymous referee, who carefully read earlier drafts of this manuscript and suggested many improvements; and Chao Li and William Minicozzi, who suggested the references \cite{chodosh2022complete, Schoen75, Schoen81} which extended, whose the main result to its natural hypothesis $d \leq 7$, as an earlier draft only considered the case $d = 3$.

This research was supported by the National Science Foundation's Graduate Research Fellowship Program under Grant No. DGE-2040433.

\section{Regularity of laminations}\label{Regularity}
\subsection{Elliptic estimates on leaves}\label{Leaf estimates}
Before we prove Theorem \ref{regularity theorem} we recall some well-known estimates on minimal surfaces in normal coordinates.
See Appendix \ref{minimal surfaces} for generalities on minimal surfaces.
Let $g$ be a metric on $\RR^{d - 1}_x \times \RR_y$ satisfying the normal coordinates condition $g - I = O(K_0(|x|^2 + y^2))$, and a curvature bound $\|\Riem_g\|_{C^0} \leq K_0$, where $I$ is the identity matrix.
For a function $u \in C^1(4\Ball^{d - 1})$, let
$$Pu(x) = F(x, u(x), \nabla u(x), \nabla^2 u(x)) = 0$$
be the minimal surface equation.
Then by \cite[(7.21)]{colding2011course}, the coefficient $F$ has the form
$$F(x, y, \xi, H) = \tr H + O((K_0(|x| + |y|) + |\xi|)(1 + |H|))$$
at least if $K_0(|x| + |y|) + |\xi|$ is small enough.
Thus if $\|u\|_{C^0(4\Ball^{d - 1})} \leq 10$ and $K_0$ and $\|\dif u\|_{C^0(4\Ball^{d - 1})}$ are small enough, the minimal surface equation is uniformly elliptic, so that by Schauder estimates \cite[Theorem 6.2]{gilbarg2015elliptic}, for any $r \geq 0$,
\begin{equation}\label{norms on uk}
\|u\|_{C^r(3\Ball^{d - 1})} \lesssim_r 1.
\end{equation}

\begin{lemma}
Suppose that $u_2 \geq u_1$ satisfy $Pu_1 = Pu_2 = 0$ on $4\Ball^{d - 1}$ and $v := u_2 - u_1$.
Then if $\|u\|_{C^0(4\Ball^{d - 1})} \leq 10$ and $K_0$ and $\|\dif u\|_{C^0(4\Ball^{d - 1})}$ are small enough,
\begin{equation}\label{Schauder Harnack}
	\|\dif v\|_{C^0(\Ball^{d - 1})} \lesssim \sup_{2\Ball^{d - 1}} v \lesssim \inf_{\Ball^{d - 1}} v.
\end{equation}
\end{lemma}
\begin{proof}
By the proof of \cite[Theorem 7.3]{colding2011course}, there exists a linear partial differential operator $Q$ such that $Qv = 0$, and if $\|u\|_{C^0(4\Ball^{d - 1})} \leq 10$ and $K_0$ and $\|\dif u\|_{C^0(4\Ball^{d - 1})}$ are small enough, then on $3\Ball^{d - 1}$, $Q$ is uniformly elliptic, and the coefficients are bounded in $C^1$.
The claim now follows from Schauder estimates and the Harnack inequality \cite[Corollary 9.25]{gilbarg2015elliptic}.
\end{proof}

For the remainder of \S\ref{Regularity}, we fix a constant $K_0$ satisfying the hypotheses of the above lemma.

The Harnack inequality (\ref{Schauder Harnack}) implies the maximum principle for minimal hypersurfaces, which we shall need later in the paper.
This maximum principle is well-known on euclidean space \cite[Corollary 1.28]{colding2011course}, but appears to be folklore on Riemannian manifolds.

\begin{proposition}[maximum principle]\label{maximum principle}
Let $N_1, N_2$ be complete, connected minimal hypersurfaces in a geodesic ball, such that $N_1 \cap N_2$ is nonempty, and such that $N_2$ lies on one side of $N_1$.
Then $N_1 = N_2$.
\end{proposition}
\begin{proof}
Since $N_2$ lies on one side of $N_1$, for every $p \in N_1 \cap N_2$, $N_2$ must be tangent to $N_1$ at $p$.
By rescaling, we may assume that $\|\Riem_g\|_{C^0} \leq K_0$.
Working in normal coordinates at $p$, chosen so that $N_1$ is the graph of a function $u_1$ on some ball $B(p, r)$, it follows (possibly after shrinking $r$) that $N_2$ is the graph of a function $u_2 \geq u_1$.
Then $v := u_2 - u_1$ satisfies $v(0) = 0$, so by (\ref{Schauder Harnack}), for some $\delta > 0$ depending on $r$, and every $x$ such that $|x| < \delta$, $v(x) = 0$.
So the germs of $N_1$ and $N_2$ at $p$ agree, and the result follows by a bootstrapping argument.
\end{proof}

\subsection{A preliminary choice of coordinates}
We now construct normal coordinates in which the leaves of $\lambda$ are $C^1$-close to hyperplanes $\{y = y_0\}$.
The utility of this fact is that, if $f: \RR^{d - 1}_x \to \RR_y$, and its graph has normal vector $\normal$, then
\begin{equation}\label{nabla as a normal}
	\normal = \frac{\partial_y f - \nabla f}{\sqrt{1 + |\nabla f|^2}}.
\end{equation}
So if $Pf = 0$, then the leaves of $\lambda$ are minimal graphs which are small in $C^1$ and so we may apply (\ref{Schauder Harnack}) uniformly among all of the leaves at once.

A similar result was proven by \cite{Solomon86} (without the quantitative dependence) using the regularity theory for integral flat convergence of minimal currents \cite[Theorem 5.3.14]{federer2014geometric}.
We did not do this because it does not seem particularly easy to recover quantitative bounds from the highly general theory of \cite[Chapter 5]{federer2014geometric}.

\begin{lemma}\label{existence of tubes}
	Let $N$ be an embedded $C^2$ hypersurface in $\RR^d = \RR^{d - 1}_x \times \RR_y$ which is tangent to $\{y = 0\}$ at the origin.
	If $\|\Two_N\|_{C^0} \leq \frac{1}{8}$, then the connected component of $N \cap B(0, 1)$ containing $0$ is the graph over $\{y = 0\}$ of a function $f$ with
	$$|f(x)| \leq \|\Two_N\|_{C^0} |x|^2.$$
\end{lemma}
\begin{proof}
	Near $0$, $N$ can be represented a graph $\{y = f(x)\}$, since it is tangent to $\{y = 0\}$.
	This representation is valid on the component of the set $\{|\nabla f(x)| < \infty\}$ containing $0$, and it is related to the unit normal by (\ref{nabla as a normal}).
	Rearranging (\ref{nabla as a normal}) and taking derivatives,
	$$-\nabla^2 f(x) = \frac{\nabla \normal(x, f(x)) \cdot (\partial_x \otimes \partial_x + \nabla f(x) \otimes \partial_y)}{\sqrt{1 + |\nabla f(x)|^2}} - \frac{\nabla^2 f(x) \cdot (\nabla f(x) \otimes \normal(x, f(x)))}{(1 + |\nabla f|^2)^{3/2}}.$$
	Here $-\nabla^2$ denotes the negative Hessian, not the Laplacian.
	Since
	$$|\partial_x \otimes \partial_x + \nabla f(x) \otimes \partial_y| \leq \sqrt{1 + |\nabla f(x)|^2},$$
	and $\nabla \normal = \Two_N$, we conclude
\begin{equation}\label{bound Hessian by Two}
	|\nabla^2 f(x)| \leq |\Two_N(x, f(x))| + |\nabla^2 f(x)| |\nabla f(x)|.
\end{equation}
	In order to control the error terms in (\ref{bound Hessian by Two}), we make the \dfn{bootstrap assumption}
\begin{equation}\label{bootstrap}
	|\nabla f(x)| \leq \frac{1}{2},
\end{equation}
	which is at least valid in some small neighborhood $B_R$ of $0$ since (\ref{nabla as a normal}) and the fact that $N$ is tangent to $\{y = 0\}$ at $0$ imply that $\nabla f(0) = 0$.
	By (\ref{bound Hessian by Two}),
$$|\nabla^2 f(x)| \leq 2|\Two_N(x, f(x))|,$$
	and integrating this inequality one obtains for $|x| \leq R$ that
\begin{equation}\label{closed bootstrap}
	|\nabla f(x)| \leq 2|\Two_N(x, f(x))| |x| \leq \frac{1}{4}.
\end{equation}
	In particular, since $\nabla f \in C^1$, either $R \geq 1$ or there exists $R' > R$ such that the bootstrap assumption (\ref{bootstrap}) is valid on $B_{R'}$.
	Therefore (\ref{bootstrap}) is valid with $R = 1$.
	Integrating (\ref{closed bootstrap}), we obtain the desired conclusion.
\end{proof}

\begin{lemma}\label{lams have C0 fields}
	Suppose that $\delta > 0$ is small enough depending on $K$.
	Then there exists $r = r(\delta, K, i, A) > 0$ such that for every disjoint family of hypersurfaces $\mathcal S$ satisfying the curvature bound (\ref{curvature bound in regularity}) and every $p \in \bigcup_{N \in \mathcal S} N$, we can choose normal coordinates $(x, y) \in \RR^{d - 1} \times \RR$ based at $p$ so that
\begin{equation}\label{normal is basically dy}
	\sup_{N \in \mathcal S} \|\normal_\lambda - \partial_y\|_{C^0(B(p, r))} \leq \delta.
\end{equation}
\end{lemma}
\begin{proof}
Consider normal coordinates $(x, y)$ based at $p$, and write $\Two_N'$ for the second fundamental form of $N \in \mathcal S$ taken with respect to the euclidean metric from those coordinates, $\normal_N'$ the euclidean normal, $\nabla'$ the euclidean Levi-Civita connection, and $\Gamma$ the Christoffel symbols.
In particular, since $\normal_N^\flat$ is the conormal and satisfies $\normal_N^\flat = (\normal_N')^\flat/|\normal_N'|$, 
$$\Two_N' = \nabla' (\normal_N')^\flat = (\nabla - \Gamma) |\normal_N'| \normal_N^\flat = |\normal_N'| (\Two_N - |\normal_N'| \Gamma \otimes \normal_N^\flat) + \nabla' \normal_N' \otimes \normal_N^\flat.$$
Using estimates on normal coordinates we conclude that for every $0 < s < i$ and some absolute $C > 0$,
$$\|\Two_N'\|_{B(p, s)} \leq A + CKs.$$
After rescaling we may assume that $A \leq 1/16$, $K \leq 1/(32C)$, and $i \geq 2$, so $\|\Two_N'\|_{C^0(B(p, 2))} \leq 1/8$.
Then we apply Lemma \ref{existence of tubes}: for $q \in N \cap B(p, 1)$, $B(q, 1) \subseteq B(p, 2)$, and $\tilde x$ the euclidean coordinate on $T_q N$ induced by the normal coordinates $(x, y)$, $N \cap B(q, 1)$ is the graph of a function $f$ on $T_q N$ satisfying
\begin{equation}\label{living in a tube}
|f(\tilde x)| \leq A|\tilde x|^2.
\end{equation}
Here, and for the remainder of this proof, we use $|\cdot|$ to mean the euclidean metric only.

Let $0 < r < s\delta^2$ for some small absolute $s > 0$ to be chosen later.
Choose some $N \in \mathcal S$ and $q \in B(p, r) \cap N$, and choose a normal coordinates system $(x, y)$ such that $\partial_y$ is a scalar multiple of $\normal_N(q)$ (where we use the euclidean metric to compare vectors in different tangent spaces).
Assume towards contradiction that (\ref{normal is basically dy}) fails.
Then there exists $N' \in \mathcal S$ and $q' \in B(p, r) \cap N'$ such that
$$|\normal_{N'}(q') - \partial_y| > \delta.$$
In particular, since $|\normal_{N'}(q')| = 1 + O(r^2)$ and $|\partial_y| = 1$, the angle $\theta$ between these two vectors is given by the law of cosines as 
$$1^2 + (1 + O(r^2))^2 - 1(1 + O(r^2)) \cos \theta = |\normal_{N'}(q') - \partial_y|^2$$
which can be neatly estimated for $s$ small enough as
$$\cos \theta < 1 - \frac{\delta^2}{2} + O(r^2) \leq 1 - \frac{\delta^2}{4}.$$
But $\theta$ is the angle between the tangent planes $T_q N$ and $T_{q'} N'$.
We consider the triangle $\Delta(q, q', r)$ where $r$ is a point of intersection of $P := T_q N$ and $P' := T_{q'} N'$, so again by the law of cosines, if $\alpha := |q - r|$ and $\beta := |q' - r|$,
$$\alpha^2 + \beta^2 - 2\alpha\beta \cos \theta = |q - q'|^2 \leq r^2.$$
By Young's inequality, it follows that 
$$r^2 \geq (\alpha^2 + \beta^2)(1 - \cos \theta) > (\alpha^2 + \beta^2) \frac{\delta^2}{4}$$
or in other words 
$$\alpha^2 + \beta^2 < \frac{4r^2}{\delta^2} < 4s^2 \delta^2$$
which means for $\delta$ small that $\max(\alpha, \beta) < 2c\delta < s/4$.
Hence $P, P'$ intersect in $B(p, s/4 + r) \subseteq B(p, s/2)$.

Now consider the tubes $\mathcal T, \mathcal T'$ of all points which are within $s^2/16$ of $P, P'$.
Since $P, P'$ intersect in $B(p, s/2)$, if $s$ is small, any graphs over $P, P'$ in $\mathcal T, \mathcal T'$ must intersect in $B(p, s)$.
In particular we can take $s < 1$ and conclude from (\ref{living in a tube}) that $N, N'$ are not disjoint, contradicting the definition of $\mathcal S$.
\end{proof}

\subsection{Proof of Theorem \ref{regularity theorem}}
Fix $\delta > 0$ to be chosen later but depending only on $i, K, d$, and $P \in M$.
Let $\normal$ be the normal vector to the hypersurfaces in $\mathcal S$.
By Lemma \ref{lams have C0 fields}, if $\delta \leq \delta_*$ for some $\delta_* = \delta_*(i, K) > 0$, there exists $r = r(\delta, i, K, A) > 0$ such that $B(P, r)$ admits rescaled normal coordinates $(x, y) \in 5\Ball^{d - 1} \times (-2, 2)$ in which the curvature of the rescaled metric has a $C^0$ norm $\leq K_0$ and
\begin{equation}\label{normal is almost constant}
\|\normal - \partial_y\|_{C^0(B(P, r))} \leq \delta.
\end{equation}
Moreover,
$$|\normal \cdot \partial_y| \geq 1 - |\normal - \partial_y| \geq 1 - \delta,$$
so if $\delta \leq \delta_*$ is small enough, then in $5\Ball^{d - 1} \times (-1, 1)$, then every leaf is the graph of a function, say $u_k: 5\Ball^{d - 1} \to (-2, 2)$ where $u_k(0) = k$, and
\begin{equation}\label{derivatives small}
\|\dif u_k\|_{C^0} \leq \frac{1 - (1 - \delta)^2}{1 - \delta} \leq 3 \delta.
\end{equation}
If $r$ is chosen small enough depending on $K, K_0$, then the metric $\tilde g$ induced by $g$ on $5\Ball^{d - 1} \times (-2, 2)$ satisfies $\|\Riem_{\tilde g}\|_{C^0} \leq K_0$.
Moreover, $\|u_k\|_{C^0} \leq 2$, and $u_k$ has a minimal graph, so the elliptic estimates stated in \S\ref{Leaf estimates} apply to $u_k$ where the implied constants only depend on $d$ and not on $M, g, P, k$.

Now let $-1 < k < \ell < 1$, and let $v_{\ell k} := u_\ell - u_k$.
By (\ref{Schauder Harnack}) with $v := v_{\ell k}$, for every $x \in \Ball^{d - 1}$,
\begin{equation}\label{bound on du}
|\dif u_\ell(x) - \dif u_k(x)| \lesssim |u_\ell(x) - u_k(x)|
\end{equation}
and it follows that
\begin{equation}\label{vertical Lipschitz}
|\normal(x, u_\ell(x)) - \normal(x, u_k(x))| \lesssim |u_\ell(x) - u_k(x)|.
\end{equation}

To extend (\ref{vertical Lipschitz}) to a Lipschitz bound on $\normal$, let $X_1, X_2 \in (\Ball^{d - 1} \times (-1, 1)) \cap \supp \lambda$.
Then there exist $x_1, x_2 \in \Ball^{d - 1}$ and $k_1, k_2 \in (-1, 1)$ such that $X_i = (x_i, u_{k_i}(x_i))$.
Setting $Y := (x_2, u_{k_1}(x_2))$,
$$|\normal(X_1) - \normal(X_2)| \leq |\normal(X_1) - \normal(Y)| + |\normal(Y) - \normal(X_2)|.$$
Then by (\ref{norms on uk}) and the mean value theorem,
$$|\normal(X_1) - \normal(Y)| \lesssim |\dif u_{k_1}(x_1) - \dif u_{k_1}(x_2)| \lesssim |X_1 - Y|.$$
Moreover, by (\ref{vertical Lipschitz}),
$$|\normal(Y) - \normal(X_2)| \lesssim |u_{k_1}(x) - u_{k_2}(x)| = |Y - X_2|.$$
Since $\delta \leq \frac{1}{4}$, by (\ref{normal is almost constant}),
$$|\cos \angle(X_1 - Y, X_2 - Y)| \lesssim \delta$$
so by the law of cosines,
$$|Y - X_2|^2 + |X_1 - Y|^2 \leq |X_1 - X_2|^2 + O(\delta(|Y - X_1|^2 + |X_2 - Y|^2)).$$
The implied constant only depends on $d$, and $\delta$ is allowed to depend on $d$, so we take $\delta$ so small that
$$|Y - X_2|^2 + |X_1 - Y|^2 \leq 2|X_1 - X_2|^2.$$
In conclusion,
$$|\normal(X_1) - \normal(X_2)| \lesssim |X_1 - X_2|$$
which implies that $\normal$ is Lipschitz on $V \cap \supp \lambda$, where $V$ is the neighborhood of $P$ which was mapped to $\Ball^{d - 1} \times (-1, 1)$ by the cylindrical coordinates $(x, y)$.
In particular, $V$ contains a ball of the form $B(P, s)$, where $s$ only depends on $r$ (and $r$ only depends on $g$ and $A$).
Taking a Lipschitz extension of $\normal$ to $V$ we obtain the desired Lipschitz normal subbundle.

Following \cite[Appendix B]{ColdingMinicozziIV}, we construct the laminar flow box
\begin{align*}
	F: \tilde V &\to V \\
	(\xi, \eta) &\mapsto (\xi, f(\xi, \eta))
\end{align*}
where $\tilde V$ is an open subset of $\RR^{d - 1}_\xi \times \RR_\eta$ (and $V$ is an open subset of $\RR^{d - 1}_x \times \RR_y$) by setting
$$f(\xi, \eta) := u_\eta(\xi)$$
if $u_\eta$ exists, and if $k < \eta < \ell$ and there does not $k < \eta' < \ell$ such that $u_{\eta'}$ exists, then
$$f(\xi, \eta) := u_k(\xi) + \frac{\eta - k}{\ell - k} v_{\ell k}(\xi)$$
is the linear interpolant of $u_k$ and $u_\ell$.

By (\ref{norms on uk}), $F$ is bounded in tangential $C^\infty$, where the seminorms only depend on $d, r$.
In particular, if $z$ is a unit vector field tangent to $\{\eta = k\}$, then the pushforward
$$F_* z = (z \cdot \dif \xi^i) \partial_{x^i} + (z \cdot \dif f) \partial_y = z + (\partial_\xi f) \partial_y$$
is well-defined, and pushforwards of such vector fields span the tangent bundle of the graph of $u_k$. 
The bound
\begin{equation}\label{xiLip of f}
	\|\partial_\xi f\|_{C^0} \leq \sup_k \|\dif u_k\|_{C^0} \leq 3\delta,
\end{equation}
a consequence of (\ref{derivatives small}), establishes that
$$\|F_* z\|_{C^0} \geq \|z\|_{C^0} - 3\delta \geq \frac{\|z\|_{C^0}}{2},$$
and then 
$$\|(F_* z) F^{-1}\|_{C^0} \lesssim \|z(F \circ F^{-1})\|_{C^0} \leq \|z\|_{C^0} \leq 2\|F_* z\|_{C^0}.$$
Since $z$ was arbitrary we conclude that $F^{-1}$ is bounded in tangential $C^1$, hence in tangential $C^\infty$ by the inverse function theorem, where the seminorms only depend on $d, r$.

It remains to show that $F$ is a Lipschitz isomorphism.
To do this, we first claim that $\Lip(f) \lesssim 1$ (where the implicit constant only depends on $d$).
In the $\xi$ direction, we use (\ref{xiLip of f}).
If $-1 < k < \ell < 1$, then by (\ref{bound on du}) and (\ref{Schauder Harnack}),
\begin{equation}\label{f lip}
	|f(\xi, k) - f(\xi, \ell)| \lesssim |u_k(\xi) - u_\ell(\xi)| \lesssim \ell - k.
\end{equation}
This shows that $f$ is Lipschitz in the $\eta$ direction on the leaves with constant only dependent on $d$, and hence on its entire domain by linear interpolation, proving the claim.
We can then estimate using (\ref{f lip})
$$|F(\xi_1, \eta_1) - F(\xi_2, \eta_2)| \lesssim |\xi_1 - \xi_2| + \Lip(f)(|\xi_1 - \xi_2| + |\eta_1 + \eta_2|)$$
so that $\Lip(F) \lesssim 1 + \Lip(f) \lesssim 1$.

To obtain a bound on $\Lip(F^{-1})$, we observe that
\begin{equation}\label{F is coLip in xi}
|\xi_1 - \xi_2|^2
\leq |\xi_1 - \xi_2|^2 + |f(\xi_1, \eta_1) - f(\xi_2, \eta_1)|^2 
= |F(\xi_1, \eta) - F(\xi_2, \eta)|^2.
\end{equation}
By Harnack's inequality with $\eta_1 = k$ and $\eta_2 = \ell$, or $k \leq \eta_1 < \eta_2 \leq \ell$ if $\eta_1, \eta_2$ lie in the plaque between leaves $k, \ell$,
$$\frac{|f(\xi_1, \eta_1) - f(\xi_1, \eta_2)|}{|\eta_1 - \eta_2|} \gtrsim \frac{v_{\ell k}(\xi_1)}{\ell - k} \gtrsim \frac{v_{\ell k}(0)}{\ell - k} = 1$$
whence by the mean value theorem, (\ref{F is coLip in xi}), and (\ref{derivatives small}),
\begin{align*}
	|\eta_1 - \eta_2| 
	&\lesssim |f(\xi_1, \eta_1) - f(\xi_1, \eta_2)| \\
	&\leq |f(\xi_1, \eta_1) - f(\xi_2, \eta_2)| + \|\partial_\xi f\|_{C^0} |\xi_1 - \xi_2| \\
	&\leq 2 |F(\xi_1, \eta_1) - F(\xi_2, \eta_2)|.
\end{align*}
Thus
$$|F(\xi_1, \eta_1) - F(\xi_2, \eta_2)| \gtrsim |\xi_1 - \xi_2|^2 + |\eta_1 - \eta_2|^2.$$
It follows that $\Lip(F^{-1}) \lesssim 1$, so $F$ is a Lipschitz isomorphism such that $\Lip(F)$ and $\Lip(F^{-1})$ only depend on $d$.

Finally, we compose $F$ with the change of coordinates at the start of this proof to obtain a laminar flow box in a small neighborhood of $(0, 0)$ whose image has radius $O(r)$, and whose Lipschitz constants are comparable to $O(r^{-1})$.


\section{Ruelle-Sullivan currents and functions of least gradient}\label{Prelims}
\subsection{Ruelle-Sullivan currents}\label{RS prelims}
Let $(\lambda, \mu)$ be a measured oriented lamination.
Then the Ruelle-Sullivan current $T_\mu$ is a well-defined closed $1$-current, and we may lift $T_\mu$ to the universal cover $\tilde M$, where it is exact \cite[\S8]{daskalopoulos2020transverse}.\footnote{The reference deals with geodesic laminations in a closed hyperbolic surface, but the arguments work for any tangentially $C^1$ measured oriented lamination.}
Moreover, $T_\mu$ has an intrinsic definition as the unique $1$-current with a certain polar decomposition.
To be more precise, recall that $\mu$ defines a measure on $\supp \lambda$: in each flow box $F_\alpha$, an open set $U$ has measure
\begin{equation}\label{transverse measure of an open set}
\mu(U) := \int_{K_\alpha} \mathcal H^{d - 1}(F_\alpha(\{k\} \times J) \cap U) \dif \mu_\alpha(k).
\end{equation}

\begin{lemma}
For a measured oriented lamination $(\lambda, \mu)$, with Lipschitz normal vector $\normal_\lambda$, the polar decomposition of $T_\mu$ is
\begin{equation}\label{polar ruelle sullivan}
T_\mu = \normal_\lambda \mu.
\end{equation}
\end{lemma}
\begin{proof}
For an open set $U \subseteq M$ in a flow box $F_\alpha$, the total variation satisfies
$$\int_U \star |T_\mu| = \sup_{\|\varphi\|_{C^0} \leq 1} \int_{K_\alpha} \int_{\{k\} \times J} \varphi \dif \mu_\alpha(k)$$
where the supremum ranges over $(d - 1)$-forms $\varphi$ with compact support in $U$.
However, $\star \normal_\lambda^\flat$ is the Riemannian measure on $F_\alpha(\{k\} \times J)$, so
$$\int_{\{k\} \times J} \varphi \leq \int_{\{k\} \times J} (F_\alpha^{-1})^*(\star \normal_\lambda^\flat).$$
Since $\|\normal^\lambda\|_{C^0} = 1$, it follows that a sequence of cutoffs of $\star \normal_\lambda^\flat$ to more and more of $U$ is a maximizing sequence.
Therefore $\normal_\lambda$ is the polar part of (\ref{polar ruelle sullivan}), and
$$\int_U \star |T_\mu| = \int_{K_\alpha} \int_{\{k\} \times J} (F_\alpha^{-1})^*(1_U \star \normal_\lambda^\flat) \dif \mu_\alpha(k).$$
The inner integral is the Riemannian measure of $F_\alpha(\{k\} \times J) \cap U$, so by (\ref{transverse measure of an open set}), $|T_\mu| = \mu$.
\end{proof}

The above computation motivates the definition of Ruelle-Sullivan current of a \emph{nonorientable} lamination.
To be more precise, if $\lambda$ is a nonorientable lamination with normal vector field $\normal_\lambda$, then we can view $\normal_\lambda$ as a section of a line bundle $L$ over $M$ of structure group $\ZZ/2$.
We can then define $T_\mu$ to be $\normal_\lambda \mu$, which makes sense as a distributional section of $L$, and can be tested against any continuous $(d - 1)$-form on $M$ whose support is contained in a trivializing chart of $L$.
In particular, we shall speak of the Ruelle-Sullivan current of any measured lamination, even if it is nonorientable.

\subsection{Functions of least gradient}\label{least gradient formulation}
Recall that in our convention, a function $u$ on $M$ has least gradient, if for every compactly supported perturbation $v$,
$$\int_{\supp v} \star |\dif u| \leq \int_{\supp v} \star |\dif u + \dif v|.$$
This is a purely interior notion, so it is well-behaved even on unbounded domains \cite[\S4]{górny2021}.
The interior formulation was used in the papers \cite{Miranda66, Miranda67,BOMBIERI1969} which introduced functions of least gradient.
The results of these papers were formulated for $M$ an open subset of $\RR^d$, but the proofs go through without any changes to the more general setting when $M$ is a Riemannian manifold.
Thus we have:

\begin{theorem}\label{main thm of old paper}
Let $u \in BV_\loc(M)$ have least gradient in $M$ and $y \in \RR$. Then $1_{\{u > y\}}$ has least gradient in $M$.
In particular, if $d \leq 7$, then $\partial \{u > y\}$ is the sum of disjoint smooth area-minimizing hypersurfaces.
\end{theorem}
\begin{proof}
Let $v := 1_{\{u > y\}}$.
Then $v$ has least gradient \cite[Theorem 1]{BOMBIERI1969}, so by Theorem \ref{regularity}, $\{u > y\}$ is bounded by disjoint embedded area-minimizing hypersurfaces.
These hypersurfaces inherit an orientation from the current $\dif v$.
\end{proof}

\begin{proposition}\label{MirandaStability}
Suppose that $(u_n)$ is a sequence of functions of least gradient on $M$, and $u_n \to u$ in $L^1_\loc(M)$.
Then $u$ has least gradient and $\dif u_n \to \dif u$ vaguely.
\end{proposition}
\begin{proof}
Let $\mu(E) := \int_E \star |\dif u|$ be the total variation measure of $u$.
Then $u$ has least gradient, and for every open set $U \subseteq M$ such that $\mu(\partial U) = 0$, then $\mu(U) = \lim_{n \to \infty} \int_U \star |\dif u_n|$ \cite[Osservazione 3]{Miranda67}.
In particular, $1_U \dif u_n \to 1_U \dif u$ vaguely.
We can exhaust $M$ by open sets $U$ such that $\mu(\partial U) = 0$ (if not, then there would be uncountably many disjoint closed subsets of $M$ with positive $\mu$-measure, contradicting that $\mu$ is locally finite), so this implies that $\dif u_n \to \dif u$ vaguely.
\end{proof}

We next discuss the Dirichlet problem for functions of least gradient.
Here we assume that $\overline M = M \cup \partial M$ is a compact manifold-with-boundary and $\partial M$ is Lipschitz.
Then we have a trace map $BV(M) \to L^1(\partial M)$, denoted $u \mapsto u|_{\partial M}$, by an immediate generalization of \cite[Theorem 2.10]{Giusti77}.
Let us temporarily say that a function $u \in BV(M)$ has \dfn{least gradient up to the boundary} if, for every $v \in BV(M)$ such that $v|_{\partial M} = 0$,
$$\int_M \star |\dif u| \leq \int_M \star |\dif u + \dif v|,$$
which is the formulation of ``function of least gradient'' that became standard after the seminal paper of Sternberg, Williams, and Ziemer \cite{Sternberg1992}.
By the equivalence of (1) and (2) in \cite[Proposition 9.3]{Korte19}, the function $u$ of bounded variation has least gradient (in the interior sense) iff $u$ has least gradient up to the boundary, an equivalence we will henceforth use without comment.

Let $h \in L^1(\partial M)$ be Dirichlet data. Then there does not have to exist a function of least gradient extending $h$, so it is natural to introduce the \dfn{relaxed functional} for the data $h$,
$$\Phi_h(v) := \int_M \star |\dif v| + \int_{\partial M} |v - h| \dif \mathcal H^{d - 1}.$$

\begin{theorem}[{\cite[\S2]{Mazon14}}]\label{relaxed formulation}
Let $\overline M = M \cup \partial M$ be compact with boundary, $u \in BV(M)$, and $h := u|_{\partial M}$.
The following are equivalent:
\begin{enumerate}
\item $u$ has least gradient.
\item $u$ minimizes $\Phi_h(v)$ among all $v \in BV(M)$.
\item There exists a measurable vector field $X$ on $M$, such that:
\begin{enumerate}
\item $\|X\|_{L^\infty} \leq 1$.
\item $\nabla \cdot X = 0$.
\item $X \cdot \dif u = |\dif u|$ as Radon measures.
\end{enumerate} \label{calibration condition}
\end{enumerate}
\end{theorem}

Let us scrutinize the vector field $X$ in Theorem \ref{relaxed formulation}.
First, since $\nabla \cdot X = 0$, $X$ is $|\dif u|$-measurable and $X \cdot \dif u$ is a well-defined Radon measure, even though $X$ could be discontinuous \cite[Theorem 1.5]{Anzellotti1983}.
Similarly, if $N$ is any Lipschitz hypersurface in $M$, then the normal part of the trace, $\normal_{\partial N} \cdot X$, is a well-defined measurable function on $N$, with $\|\normal_{\partial N} \cdot X\|_{L^\infty(N)} \leq 1$ \cite[Theorem 1.2]{Anzellotti1983}.
In particular, taking $N$ to be a component of $\partial \{u > y\}$ and using Theorem \ref{main thm of old paper}, we see that $X$ must be the normal vector to $N$ (otherwise $X \cdot \normal_N^\flat < 1$, but $\normal^\flat_N = \dif u/|\dif u|$, so $X \cdot \dif u < |\dif u|$).
The vector field $X$ (or rather its Hodge dual) has a central role in \cite{BackusInfinityMaxwell1}, where these properties are studied more closely.

A difficulty with the notion of functions of least gradient is that, while one would like to think of them as the natural solution concept for the $1$-Laplacian (\ref{1Laplacian}), solutions of a PDE form a sheaf (in the sense that if $\mathcal U$ is an open cover and $u$ is a function such that $u|_U$ is a solution for every $U \in \mathcal U$, then $u$ is a solution); however, functions of least gradient do \emph{not} form a sheaf.
This is essentially equivalent to the fact that a hypersurface which is locally area-minimizing need not be area-minimizing:

\begin{example}
Let $M = \Ball^3$ be the unit ball in $\RR^3$, and $u$ be the indicator function of the region bounded by a catenoid $C$ which meets $\partial \Ball^3$ on two circles $\partial D_1, \partial D_2$, which bound disks $D_i$ of the same area $A$.
Then $\star |\dif u|$ is the surface measure $\mathcal H^{d - 1}|_C$ on $C$ by the coarea formula (\ref{coarea formula}).
Since $C$ is a minimal surface, $u$ has locally least gradient.
But we may choose $C$ so that $2A < \mathcal H^{d - 1}(C)$.
Then $u$ competes with the indicator function $v$ of the region bounded by $D_2 - D_1$, and $\int_M \star |\dif v| = 2A$, so $u$ does not have least gradient.
\end{example}

In practice, it is much easier to check whether a hypersurface is minimal than if it is area-minimizing, so the converse direction of Theorem \ref{main thm} would have limited utility if one had to assume that the leaves of $\lambda$ are area-minimizing.
These considerations are why we consider functions of locally least gradient in Theorem \ref{main thm}.

In \cite{daskalopoulos2020transverse, BackusInfinityMaxwell1}, it is necessary to study functions of least gradient in a homological sense, and we conclude this section by explaining how that notion is related to this paper.
These remarks can be ignored by the reader only interested in the present paper.

Suppose that $N$ is a closed Riemannian manifold, and we have a Riemannian covering $M \to N$.
A function $u$ on $M$ is \dfn{equivariant} for the covering if $\dif u$ descends to a $1$-form on $N$ (which we also call $\dif u$), and $u$ is \dfn{invariant} for the covering if $u$ itself descends to a function on $N$.
We say that an equivariant function $u \in BV_\loc(M)$ has \dfn{homologically least gradient} if, for every invariant $v \in BV_\loc(M)$,
$$\int_M \star |\dif u| \leq \int_M \star |\dif u + \dif v|.$$
This is the notion of least gradient used in \cite{daskalopoulos2020transverse,BackusInfinityMaxwell1}.
If $u$ is equivariant and has least gradient, then by testing against $\chi v$ where $\chi$ has compact support in a fundamental domain of the covering and $v$ is invariant, we see that $u$ has homologically least gradient.
Conversely, if $u$ has homologically least gradient, then one can use a variation of Theorem \ref{relaxed formulation} to show that there is a vector field $X$ satisfying Theorem \ref{relaxed formulation}(\ref{calibration condition}) which is equivariant, and this is enough to show that $u$ has least gradient; this is discussed more carefully in \cite{BackusInfinityMaxwell1}.

\subsection{Hausdorff limits of closed sets}
\begin{definition}[{\cite[Chapter IV]{nadler2017continuum}}]
Let $X$ be a topological space, and $(Y_n)$ a sequence of closed subsets of $X$.
\begin{enumerate}
\item The \dfn{limit inferior} $\liminf_{n \to \infty} Y_n$ is the set of all $x \in X$ such that for every open neighborhood $U \ni x$, $U \cap Y_n$ is eventually nonempty.
\item The \dfn{limit superior} $\limsup_{n \to \infty} Y_n$ is the set of all $x \in X$ such that for every open neighborhood $U \ni x$, $U \cap Y_n$ is nonempty for infinitely many $n$.
\item If $\liminf_{n \to \infty} Y_n = \limsup_{n \to \infty} Y_n$, we call that set the \dfn{Hausdorff limit} $\lim_{n \to \infty} Y_n$.
\end{enumerate}
\end{definition}

\begin{lemma}
Suppose that $\lambda_n \to \lambda$ in Thurston's geometric topology. Then
\begin{equation}\label{supports shrink in the limit}
\supp \lambda \subseteq \liminf_{n \to \infty} \supp \lambda_n.
\end{equation}
\end{lemma}
\begin{proof}
We pass to a subsequence which realizes the limit inferior in (\ref{supports shrink in the limit}).
Then by definition of a basic open set, for every $\varepsilon > 0$ we can find $n$ and $x_n$ such that $x_n \in \supp \lambda_n \cap B(x, \varepsilon)$.
Either way, we conclude (\ref{supports shrink in the limit}).
\end{proof}

\section{Application to \texorpdfstring{$1$-harmonic}{one-harmonic} functions}\label{1harmonic sec}
The purpose of this section is to prove Theorem \ref{main thm}, and explore some of its consequences.
Throughout, we shall assume that the dimension of $M$ is $2 \leq d \leq 7$.

\subsection{Estimates on the level sets}\label{level set estimates}
\begin{lemma}
There exists a continuous function $R: M \to \RR_+$ such that the following holds.
Let $p \in M$, $0 < r \leq R(p)$, and $N$ an area-minimizing hypersurface in $B(x, r)$.
Then 
\begin{equation}\label{least gradient area bound}
\mathcal H^{d - 1}(N) \leq 2d \omega_d r^{d - 1}.
\end{equation}
\end{lemma}
\begin{proof}
If $r$ is small enough depending on the geometry of $M$ near $p$, $B(p, r)$ is contractible, so $N = \partial V$ for some open set $V$.
Let $v := 1_V$.
Since $N$ is area-minimizing, $v$ has least gradient.
Let $h$ be the trace of $v$ along $\partial B(p, r)$; then $\|h\|_{L^\infty} \leq 1$ \cite[Theorem 2.10]{Giusti77}.
By (\ref{hausdorff measure is jump mass}) and Theorem \ref{relaxed formulation},
$$\mathcal H^{d - 1}(N) = \int_{B(p, r)} \star |\dif v| = \Phi_h(v) \leq \Phi_h(0) = \int_{\partial B(p, r)} |h| \dif \mathcal H^{d - 1}.$$
If $r$ is small enough depending on the curvature of $M$ near $p$, we can approximate $\partial B(p, r)$ by a round $(d - 1)$-sphere of radius $r$, and conclude
\begin{align*}
	\int_{\partial B(p, r)} |h| \dif \mathcal H^{d - 1} &\leq \mathcal H^{d - 1}(\partial B(p, r)) \leq 2d \omega_d r^{d - 1}. \qedhere
\end{align*}
\end{proof}

\begin{lemma}\label{choose balls for main thm}
There exist $C > 0$ and a continuous function $R: M \to \RR_+$ such that the following holds.
Let $p \in M$, $y \in \RR$, $0 < r \leq R(p)$, and $N$ an area-minimizing hypersurface in $B(p, 2r)$.
Then
\begin{equation}\label{least gradient curvature bound}
\|\Two_N\|_{C^0(B(p, r))} \leq \frac{C}{r}.
\end{equation}
\end{lemma}
\begin{proof}
Let $q \in N \cap B(p, r)$ (so $B(q, r) \subseteq B(p, 2r)$).
By (\ref{least gradient area bound}),
$$\mathcal H^{d - 1}(N \cap B(q, r)) \leq 2d \omega_d r^{d - 1}.$$
Since $N$ is area-minimizing, $N$ is stable.
So by \cite[pg785, Corollary 1]{Schoen81}\footnote{See also \cite[Theorem 3]{Schoen75} for an easier proof when $M$ has nonpositive curvature and dimension $d \leq 6$, or \cite[Chapter 2, \S\S4-5]{colding2011course} for a textbook treatment of a similar estimate.}
\begin{align*}
|\Two_N(q)| &\leq \|\Two_N\|_{C^0(B(q, r/2))} \lesssim \frac{1}{r}
\end{align*}
where the contribution from the curvature of $M$ can be absorbed if $R(p)$ is chosen small enough.
\end{proof}

\subsection{Proof of Theorem \texorpdfstring{\ref{main thm}}{B}: Locally least gradient implies minimal lamination}
\subsubsection{Structure of level sets}
Let $u$ be a function of locally least gradient.
The theorem is local, so we may replace $M$ by a ball $B(p, r)$, such that $u|_{B(p, r)}$ has least gradient, and such that $r \leq R(p)$, where $R$ is the function in Lemma \ref{choose balls for main thm}.

By a \dfn{level set} we mean a connected component of $\partial \{u > y\}$ for some $y \in \RR$.
By Theorem \ref{main thm of old paper}, the level sets of $u$ are area-minimizing hypersurfaces in $M$.

Let $y, z \in \RR$. If $y > z$, then $\{u > y\} \subseteq \{u > z\}$, so $\partial \{u > y\}$ lies on one side of $\partial \{u > z\}$.
By the maximum principle, Proposition \ref{maximum principle}, it follows that either $\partial \{u > y\}$ and $\partial \{u > z\}$ are disjoint, or are equal.
If we shrank $M$ enough, then by (\ref{least gradient curvature bound}), there is $A \geq 0$ such that for any level set $N$,
\begin{equation}\label{curvature bound on level sets}
\|\Two_N\|_{C^0} \leq A.
\end{equation}
By (\ref{level sets define support}), $S := \bigcup_{y \in \RR} \partial \{u > y\}$ is dense in $\overline S = \supp \dif u$.

\subsubsection{Existence of flow boxes}
Since $u$ has least gradient, $u \in L^\infty_\loc$ \cite[Theorem 4.3]{Gorny20}.\footnote{The cited proof uses the monotonicity formula for minimal hypersurfaces in $\RR^d$. This can be replaced with (\ref{monotonicity formula}) to generalize the result to arbitrary Riemannian manifolds.}
Since we are working locally, we may shrink $M$ so that $u \in L^\infty$.

A minimal hypersurface $N$ is a \dfn{generalized level set} if $N$ is locally the $C^2$ limit of level sets.

\begin{lemma}
The closed set $\overline S$ is covered by generalized level sets $N$ satisfying (\ref{curvature bound on level sets}).
\end{lemma}
\begin{proof}
Let $p \in \overline S$, and choose $p_n \in \partial \{u > y_n\}$ converging to $p$.
Since $u \in L^\infty$, $(y_n)$ must be a bounded sequence.
So after passing to a subsequence, we may assume that $(y_n)$ converges monotonically to some $y \in \RR$.

After passing to a smaller ball and rescaling, we may use (\ref{curvature bound on level sets}) to assume that the the area-minimizing hypersurfaces $N_n := \partial \{u > y_n\}$ all have small curvature in $C^0$.
Thus by Lemmata \ref{existence of tubes} and \ref{lams have C0 fields}, each hypersurface $N_n$ can locally be viewed as a graph of a function $f_n$ which is small in $C^2$, in normal coordinates centered on $p$.
By (\ref{norms on uk}), $\|f_n\|_{C^3} \lesssim 1$ and so along a subsequence, $f_n \to f$ in $C^2$, for some $f$ whose graph is an area-minimizing hypersurface $N \ni p$ satisfying the same curvature bound (\ref{curvature bound on level sets}).
\end{proof}



\begin{lemma}
Let $N, N'$ be generalized level sets.
Then $N = N'$, or $N \cap N'$ is empty.
\end{lemma}
\begin{proof}
If $N, N'$ are distinct generalized level sets which intersect at some point $p$, then after passing to a small neighborhood of $p$, we may assume that $N, N'$ are the graphs of functions $f, f'$ which are approximated in $C^2$ by sequences $f_n, f_n'$ whose graphs are level sets.
Since $p \in N \cap N'$ there exists $x$ such that $f(x) = f'(x)$.
By the maximum principle, Proposition \ref{maximum principle}, and the fact that $f, f'$ are distinct, $x$ is a saddle point of $f - f'$, so there exist $x_+, x_-$ close to $x$ with $f(x_+) > f'(x_+)$ and $f(x_-) < f'(x_-)$.
Therefore for $n$ large enough, $f_n(x_+) > f_n'(x_+)$ and $f_n(x_-) < f'_n(x_-)$, so by the intermediate value theorem there exist $x_n$ with $f_n(x_n) = f_n(x_n')$.
But the level sets of $u$ are disjoint, so this is a contradiction.
\end{proof}

So by Theorem \ref{regularity theorem}, the generalized level sets are the leaves of a Lipschitz minimal lamination $\lambda$, which by (\ref{level sets define support}) has support equal to $\overline S = \supp \dif u$.

\subsubsection{Discarding the exceptional level sets}
Our next task is to show that every generalized level set, which is not a level set, is a component of $\partial \{u < y\}$ for some $y \in \RR$.
We work in flow box coordinates $(k, x) \in I \times J$ near $p$, and write $u(k, x) = \tilde u(k)$.
Then $\tilde u \in BV_\loc(I)$, and we shall be interested in the monotonicity properties of $\tilde u$.
In Appendix \ref{BV appendix} we discuss monotonicity properties of functions of locally bounded variation.

\begin{lemma}
For each $k \in I$ there is a neighborhood $I' \subseteq I$ of $k$ such that $\tilde u|_{I'}$ is monotone.
\end{lemma}
\begin{proof}
Suppose not, and let $\sigma: \supp \dif \tilde u \to \{\pm 1\}$ be the polar part of $\dif \tilde u$.
Then, by Lemma \ref{monotonicity dichotomy}, for each $\varepsilon > 0$, possibly after replacing $\tilde u$ with $-\tilde u$, we may find $k_1 < k_2$ in $[k - \varepsilon, k + \varepsilon] \cap \supp \dif \tilde u$ such that $\sigma(k_1) = -1$ and $\sigma(k_2) = +1$, and $k_1, k_2$ are Lebesgue points of $\sigma$ with respect to $|\dif \tilde u|$.
In particular, $\{k_i\} \times J$ is a generalized level set, and in particular a minimal hypersurface.

Let $R := [k_1, k_2] \times J$.
By Theorem \ref{relaxed formulation} and the discussion after it, there is a vector field $X$ defined near $p$ such that $\nabla \cdot X = 0$, $X$ is the unit normal of $\{k = k_1\} \times J$ and $\{k = k_2\} \times J$ oriented in the direction that $u$ is increasing, and $\|X\|_{L^\infty} \leq 1$ (where lengths and angles are taken with respect to the Riemannian metric on $M$, not the euclidean metric in the flow box).
By our contradiction hypothesis, we may choose the orientation on $R$ such that for $i = 1, 2$,
$$\int_{\{k = k_i\} \times J} X \cdot \normal_{\partial R} \dif \mathcal H^{d - 1} = \mathcal H^{d - 1}(\{k = k_i\} \times J),$$
where $\mathcal H^{d - 1}, \normal_{\partial R}$ are again computed using the metric.
By the Lipschitz nature of the flow box coordinates, there is $\delta > 0$ independent of $\varepsilon$ such that $\{k = k_i\} \times J$ contains a ball of radius $\delta$.
So by (\ref{monotonicity formula}), there exists $c > 0$ independent of $\varepsilon$ such that
$$\mathcal H^{d - 1}(\{k = k_i\} \times J) \geq c\delta^{d - 1}.$$
On the other hand, $R$ has $2d$ faces, and if $\sigma$ is a $d - 1$-face of $R$ which is not a generalized level set, then $[k_1, k_2]$ is an edge of $\sigma$, so by the Lipschitz regularity, there is $C > 0$ independent of $\varepsilon$ such that $|\sigma| \leq C\varepsilon$.
So by the divergence theorem and the fact that $\|X\|_{L^\infty(\sigma)} \leq 1$ (see the discussion after Theorem \ref{relaxed formulation}),
$$0 = \int_{\partial R} X \cdot \normal_{\partial R} \dif \mathcal H^{d - 1} \geq 2c\delta^{d - 1} - 2(d - 1)C \varepsilon.$$
If $\varepsilon$ is taken small enough, this is a contradiction.
\end{proof}

Possibly after shrinking and reorienting the flow box, we may assume that $\tilde u$ is nondecreasing.
We call such a flow box an \dfn{oriented flow box}.
Under this assumption, we may partition $I$ into open intervals $I_i$ on which $\tilde u$ is constant, and a compact set $K$ on which $\tilde u$ is strictly increasing.

Let $N$ be a generalized level set, so the restriction of $N$ to the flow box is $\{k\} \times J$ for some $k \in K$.
Let $C$ be the connected component of $K$ containing $k$ (noting carefully that $K$ may be totally disconnected).
Since $C$ is a connected closed subset of $\RR$, $C$ is a closed interval.
If $k$ is not the right endpoint of $C$, then there are $k_n > k$ decreasing to $k$ with, for any $x, x' \in J$,
$$u(k_n, x) = \tilde u(k_n) > \tilde u(k) = u(k, x').$$
Therefore $N$ is a component of $\partial \{u > \tilde u(k)\}$, and in particular $N$ is a level set.

So if $N$ is not a level set, $k$ is a right endpoint of $C$.
So there are $k_n < k$ increasing to $k$ with $\tilde u(k_n) < \tilde u(k)$, so reasoning as above, $N$ is a component of $\partial \{u < \tilde u(k)\}$.

\subsubsection{Construction of Ruelle-Sullivan current}
We work in oriented flow box coordinates $(k, x) \in I \times J$, with the function $\tilde u$ defined by $\tilde u(k) = u(k, x)$ as above, and let $K := \supp \dif \tilde u$.
We obtain a measure $\mu$ on $K$ by setting, for $k_1 < k_2$,
$$\mu([k_1, k_2] \cap K) := \tilde u(k_2) - \tilde u(k_1).$$
Since the coordinates are oriented, $\mu$ is a positive Radon measure.

Suppose that $(k', x') \in I' \times J$ is a different oriented flow box coordinate system, with $\tilde u'$ and $K'$, and the transition map carries $k_1, k_2$ to $k_1', k_2'$.
If $(k, x)$ and $(k', x')$ both are coordinate representations of $p \in M$, then $\tilde u(k, x) = u(p) = \tilde u'(k', x')$, hence
$$\mu'([k_1', k_2'] \cap K') := \tilde u'(k_2') - \tilde u'(k_1') = \tilde u(k_2) - \tilde u(k_1) = \mu([k_1, k_2] \cap K).$$
So $\mu$ is transverse, and arises from the disintegration of $\star |\dif u|$ in coordinates.
It follows from (\ref{polar ruelle sullivan}) that $\dif u = \normal_\lambda |\dif u|$ is the Ruelle-Sullivan current for the given measured oriented structure on $\lambda$.


\subsection{Proof of Theorem \texorpdfstring{\ref{main thm}}{B}: Minimal lamination implies locally least gradient}
Suppose that $H^1(M, \RR) = 0$ and we are given a measured oriented minimal lamination $\lambda$ with bounded curvature, which then has a Ruelle-Sullivan current $\dif u$.
Suppose that the theorem is false, so one of the following holds:
\begin{enumerate}
\item $u$ does not have locally least gradient. \label{loc least}
\item The leaves of $\lambda$ are area-minimizing, but $u$ does not have least gradient. \label{minimizer}
\end{enumerate}
In either case, we are going to choose an open set $E$, such that $u|_E$ does not have least gradient, but every leaf $N$ of $\lambda$ satisfies that $N \cap E$ is area-minimizing.

In case (\ref{loc least}), we use Proposition \ref{minimal implies locally minimizing} to find an open cover $\mathcal U$ of $M$, such that for every $D \in \mathcal U$, and every leaf $N$ of $\lambda$ which meets $D$, $N \cap D$ is area-minimizing in $D$.
Since $u$ does not have locally least gradient, there exists $E \in \mathcal U$ such that $u|_E$ does not have least gradient.

In case (\ref{minimizer}), let $E := M$.

Since $u|_E$ does not have least gradient, there exists $v \in BV_\cpt(E)$ such that
\begin{equation}\label{not least gradient compact support}
\int_E \star |\dif u| > \int_E \star |\dif u + \dif v|.
\end{equation}
In particular, there exists a collar neighborhood $F$ of the boundary such that $v|_F = 0$.
Then for each $y \in \RR$,
$$\partial \{u > y\} \cap F = \partial \{u + v > y\} \cap F,$$
so that $1_{\{u + v > y\}} - 1_{\{u > y\}}$ has compact support in $E$.
Since $\partial \{u > y\}$ is area-minimizing in $E$, $1_{\{u > y\}}$ has least gradient in $E$.
So we may estimate using the coarea formula (\ref{coarea formula}):
\begin{align*}
\int_E \star |\dif u| &= \int_{-\infty}^\infty \int_E \star |\dif 1_{\{u > y\}}| \dif y \leq \int_{-\infty}^\infty \int_E \star |\dif 1_{\{u + v > y\}}| \dif y = \int_E \star |\dif u + \dif v|
\end{align*}
which is a contradiction of (\ref{not least gradient compact support}).
This completes the proof of Theorem \ref{main thm}.

\subsection{The G\'orny decomposition}\label{1harmonic apps}

We now consider an analogue of the G\'orny decomposition \cite[Theorem 1.2]{górny2017planar} of a function of least gradient.
Recall that a \dfn{Cantor function} is a continuous function whose exterior derivative is mutually singular with Lebesgue measure.
A \dfn{jump function} is a function $u \in BV_\loc$ such that $\dif u$ equals its own jump part as in \cite[Definition 3.91]{Ambrosio2000FunctionsOB}.
In general, it is not possible to decompose a function $u$ of bounded variation into an absolutely continuous (that is, $W^{1, 1}_\loc$) part, a Cantor part, and a jump part \cite[Example 4.1]{Ambrosio2000FunctionsOB}.
G\'orny showed that for a function of least gradient on euclidean space, such a decomposition exists.

We give a new proof using Theorem \ref{main thm} which applies on curved domains.
As a byproduct, we obtain a new proof of the continuity of jump-free functions of least gradient \cite[Theorem 4.1]{HakkarainenKorteLahtiShanmugalingam+2015}, though their result holds in the higher generality that $M$ is a metric measure space.

\begin{proposition}
Let $u$ be a function of locally least gradient, and suppose that $H^1(M, \RR) = 0$. Then there exists a decomposition of $u$ into functions of locally least gradient
$$u = u_{ac} + u_C + u_j,$$
with mutually singular exterior derivatives, such that $u_{ac} \in W^{1, 1}_\loc(M) \cap C^0(M)$, $u_C$ is a Cantor function, and $u_j$ is a jump function.
Up to addition of additive constants, this decomposition is unique.
\end{proposition}
\begin{proof}
We work in flow box coordinates $(k, x) \in I \times J$ on some open set $V_\alpha \subseteq M$.
We may assume that $V_\alpha$ is so small that $u|_{V_\alpha}$ has least gradient.
In the flow box coordinates, $u(k, x) = \tilde u^\alpha(k)$ for some $\tilde u_\alpha: I \to \RR$. 
Since $\dif \tilde u^\alpha$ is the transverse measure, it is a Radon measure, so $\tilde u^\alpha$ has bounded variation and hence we have the Lebesgue decomposition on an interval \cite[Corollary 3.33]{Ambrosio2000FunctionsOB} 
$$\tilde u^\alpha = \tilde u^\alpha_{ac} + \tilde u^\alpha_C + \tilde u^\alpha_j$$
where $\tilde u^\alpha_{ac} \in W^{1, 1}_\loc(I)$, $\tilde u^\alpha_C$ is a Cantor function, and $\tilde u_j^\alpha$ is a jump function.
This decomposition is unique up to an addition of constants, and induces a decomposition of $\dif \tilde u^\alpha$ into mutually singular measures.
We then write $u^\alpha_\sigma(k, x) = \tilde u^\alpha_\sigma(k)$ to obtain a function on $V_\alpha \subseteq M$, where $\sigma \in \{ac, C, j\}$.
The functions $u^\alpha_{ac}, u^\alpha_C, u^\alpha_j$ are $W^{1,1}_\loc$, Cantor, and jump respectively, since Lipschitz isomorphisms preserve these conditions.
Moreover, since $\tilde u^\alpha_{ac}$ and $\tilde u^\alpha_C$ are jump-free functions of bounded variation on an interval, they are continuous; hence $u^\alpha_{ac}$ and $u^\alpha_C$ are continuous as well.

We next claim that $\dif u^\alpha_{ac}, \dif u^\alpha_C, \dif u^\alpha_j$ are mutually singular.
Since $\dif \tilde u^\alpha_{ac}, \dif \tilde u^\alpha_C, \dif \tilde u^\alpha_j$ are mutually singular, we have a decomposition
$$I = I_{ac} \sqcup I_C \sqcup I_j$$
such that for $\sigma \neq \tau$, $I_\tau$ is a $\dif \tilde u^\alpha_\sigma$-null set.
Applying Fubini's theorem, the same decomposition holds for $\dif u^\alpha_\sigma$ and $I \times J \cong V_\alpha$, implying mutual singularity of the $\dif u^\alpha_\sigma$.

We now claim that $u^\alpha_\sigma$ have least gradient on $V_\alpha$; this step is identical to the analogous step in \cite{górny2017planar}.
To ease notation we do this for $\sigma = j$; the other cases are similar.
If the claim fails, then there is some $v \in BV_\cpt(V_\alpha)$ such that $\int \star |\dif u^\alpha_j| > \int \star |\dif (u^\alpha_j + v)|$.
But if so, then by mutual singularity,
\begin{align*}
\int_{V_\alpha} \star |\dif u| &= \int_{V_\alpha} \star (|\dif u^\alpha_j| + |\dif u^\alpha_C| + |\dif u^\alpha_{ac}|) 
>  \int_{V_\alpha} \star (|\dif u^\alpha_j + \dif v| + |\dif u^\alpha_C| + |\dif u^\alpha_{ac}|) \\
&\geq \int_{V_\alpha} \star |\dif u + \dif v|,
\end{align*}
which contradicts that $u$ has least gradient.

Finally we glue the local decompositions together.
The measure-preserving property of transition maps and the uniqueness of the Lebesgue decomposition implies that
$$\dif u^\alpha_\sigma|_{V_\alpha \cap V_\beta} = \dif u^\beta_\sigma|_{V_\alpha \cap V_\beta}.$$
As closed currents form a sheaf, it follows that there exist unique closed currents $\dif u_\sigma$ on all of $M$ such that $\dif u_\sigma|_{V_\alpha} = \dif u^\alpha_\sigma$.
Since $H^1(M, \RR) = 0$, $\dif u_\sigma$ has an antiderivative $u_\sigma$, which has locally least gradient, since $u_\sigma|_{V_\alpha} = u_\sigma^\alpha$ has least gradient.
\end{proof}

\section{Compactness of the space of laminations}\label{CompactnessSec}
In this section we prove Theorem \ref{compactness theorem}, the compactness theorem.
We then apply it to explore the implications between the different modes of convergence.


\subsection{Proof of Theorem \texorpdfstring{\ref{compactness theorem}}{C}}
\subsubsection{Construction of the limiting flow box}
Let $P \in M$, and let $(\lambda_n)$ be a sequence of minimal laminations of bounded curvature, such that every leaf of every lamination meets a compact set.

By Theorem \ref{regularity theorem}, there exist $r > 0$ and $L \geq 1$ such that for every large $n \in \NN$, $B(P, r)$ is contained in the image of a flow box $F_n$ for $\lambda_n$ with Lipschitz constant $L$, such that $F_n(0, 0) = P$, and such that $F_n$ is bounded in tangential $C^\infty$ independently of $n$.
By the Arzela-Ascoli theorem, along a subsequence $F_n \to F$ in $C^0$ for some map $F: I \times J \to B(P, r)$ and some $I \subseteq \RR$, $J \subseteq \RR^{d - 1}$, such that on the image $V$ of $F$, we also have the convergence $F_n^{-1} \to F^{-1}$.
Moreover, $F(0, 0) = P$, so that $F: I \times J \to V$ is a homeomorphism onto a set which contains $P$.
Since
$$\max(\Lip(F), \Lip(F^{-1})) \leq \limsup_{n \to \infty} \max(\Lip(F_n), \Lip(F_n^{-1})) \leq L,$$
it follows that $\max(\Lip(F), \Lip(F^{-1})) \leq L$, and for any $\theta \in (0, 1)$,
\begin{align*}
	\|F - F_n\|_{C^\theta}
	&\leq \Lip(F - F_n)^\theta \|F - F_n\|_{C^0}^{1 - \theta} \leq (2L)^\theta \|F - F_n\|_{C^0}^{1 - \theta}.
\end{align*}
It follows that $F_n \to F$ in $C^\theta$, hence in $C^{1-}$, and similarly for $F^{-1}$.

Since $P$ was arbitrary, it follows that we can find laminar atlases $(F_\alpha^n, K_\alpha^n)$ for each large $n \in \NN$ such that $(F_\alpha^n)$ is bounded in tangential $C^\infty$, $F_\alpha^n \to F_\alpha$ and $(F_\alpha^n)^{-1} \to (F_\alpha)^{-1}$ in $C^{1-}$, where the images of $F_\alpha$ and $F_\alpha^n$ are an open cover $(U_\alpha)$ of $M$ independent of $n$, and $(F_\alpha)$ satisfies the usual transition relations, and $F_\alpha$ is a Lipschitz isomorphism.

We fix the data $L$, $(F_\alpha^n, K_\alpha^n)$, $U_\alpha$, and $F_\alpha$ for the remainder of the proof of Theorem \ref{compactness theorem}.


\subsubsection{Construction of the limiting lamination}
We now construct the limiting lamination.
We employ the Hausdorff hyperspace $\Hypspace I$ of closed subsets of $I$ to accomplish this.
Since $I$ is a compact metric space, so is $\Hypspace I$ \cite[Theorem 4.17]{nadler2017continuum}, so we may diagonalize so that for every $\alpha$, either $K^n_\alpha \to K_\alpha$ for some nonempty $K_\alpha$ in the Hausdorff distance on $I$, or there exists $n^*(\alpha)$ such that for every $n \geq n^*(\alpha)$, $K_\alpha^n$ is empty (in which case we define $K_\alpha = \emptyset$).

By assumption, there is a compact set $E \subseteq M$ such that for every $n$, every leaf of $\lambda_n$ meets $E$.
Since $E$ is compact, there exists a finite set $A_E \subseteq A$ such that $E \subseteq \bigcup_{\alpha \in A_E} U_\alpha$.

\begin{lemma}\label{label sets are nonempty}
	There exists $\alpha \in A_E$ such that $K_\alpha$ is nonempty.
\end{lemma}
\begin{proof}
	Suppose not; then for
	$$n \geq \max_{\alpha \in A_E} n^*(\alpha)$$
	and $\alpha \in A_E$, $K_\alpha^n = \emptyset$, so no leaves of $\lambda_n$ meet $U_\alpha$, and hence no leaves of $\lambda_n$ meet $E$.
	This is a contradiction since $\lambda_n$ has a leaf.
\end{proof}

In each flow box $F_\alpha$ with $K_\alpha$ nonempty, we thus have the leaves of a lamination, namely $K_\alpha \times J$.
We now check the transition relations to ensure that they glue to a global lamination; this is straightforward but we include it for completeness.

Thus let $\psi_{\alpha \beta}$ and $\psi_{\alpha \beta}^n$ be the transition maps, thus $\psi_{\alpha \beta}^n$ induces a map
$$\psi_{\alpha \beta}^n: K_\alpha^n \to K_\beta^n.$$
By convergence of $(F_\alpha^n)$, $\psi_{\alpha \beta}$ induces a map $K_\alpha \to K_\beta$.

\begin{definition}
	A \dfn{cocycle of labels} $(k_\alpha)_{\alpha \in A'}$ is a set $A' \subseteq A$ and an element of $\prod_{\alpha \in A'} K_\alpha$, such that:
\begin{enumerate}
	\item The cocycle condition: $k_\beta = \psi_{\alpha \beta}(k_\alpha)$ for $\alpha, \beta \in A'$.
	\item Closure under transition maps: For every $\alpha \in A'$, if $\psi_{\alpha \beta}(k_\alpha)$ is well-defined, then $\beta \in A'$.
\end{enumerate}
\end{definition}

\begin{lemma}
	Every cocycle of labels $(k_\alpha)_{\alpha \in A'}$ defines a complete minimal hypersurface $N$ such that
	$$N \cap U_\alpha = F_\alpha(\{k_\alpha\} \times J).$$
\end{lemma}
\begin{proof}
We have the cocycle condition
$$(N \cap U_\alpha) \cap U_\beta = (N \cap U_\beta) \cap U_\alpha$$
which follows from the fact that
\begin{align*}
F_\alpha(\{k_\alpha\} \times J) \cap U_\beta
&= F_\beta(\psi_{\alpha \beta}(\{k_\beta\} \times J)) \cap U_\alpha \cap U_\beta \\
&= F_\beta(\psi_{\alpha \beta}(\{k_\beta\} \times J)) \cap U_\alpha.
\end{align*}
From the cocycle condition, it follows that $N$ honestly defines a Lipschitz hypersurface in $M$, which is complete in $\bigcup_{\alpha \in A'} U_\alpha$.
If $\overline N$ intersects $U_\alpha$ for some $\alpha \notin A'$, then $N$ intersects $U_\beta$ for some $\beta \in A'$ so that $U_\beta \cap U_\alpha \cap \overline N$ is nonempty.
But then $\psi_{\beta \alpha}(k_\beta)$ must be defined, so $\alpha \in A'$, a contradiction.
Therefore $N$ is complete in $M$.

For each $\alpha \in A'$, there exist $k_\alpha^n \in K_\alpha^n$ such that $k_\alpha^n \to k_\alpha$.
Let $N_n \cap U_\alpha$ be the corresponding minimal hypersurface in $M$.
Since $F_\alpha^n \to F_\alpha$ in $C^0$, for every neighborhood $V$ of $N \cap U_\alpha$, there exists $n_*$ such that if $n \geq n_*$ then $N_n \cap U_\alpha \subset V$.
However, the minimal hypersurfaces $N \cap U_\alpha$ are assumed to have bounded curvature, and satisfy $\partial (N \cap U_\alpha) \subset \partial U_\alpha$.
Moreover,
$$\mathcal H^{d - 1}(N \cap U_\alpha) \leq L^{d - 1} \mathcal H^{d - 1}(\{k_\alpha^n\} \times J) \lesssim L^{d - 1}$$
(where the measure of $\{k_\alpha^n\} \times J$ was taken with respect to the euclidean metric, hence is bounded by an absolute constant).
So by Theorem \ref{compactness for minimal surfaces}, as $n \to \infty$, $N_n \cap U_\alpha$ converges in $C^\infty$ to a minimal hypersurface $N' \subset U_\alpha$ with $\partial N' \subset \partial U_\alpha$.
Then $N' \subset V$ for every neighborhood $V$ of $N \cap U_\alpha$, so $N' \subseteq N \cap U_\alpha$.
The completeness of $N'$ then implies that $N = N \cap U_\alpha$.
\end{proof}

\begin{lemma}
	Let $\lambda$ be the lamination with laminar atlas $(F_\alpha, K_\alpha)$.
	Then $\lambda$ is well-defined and minimal.
\end{lemma}
\begin{proof}
Since 
$$\supp \lambda \cap U_\alpha = K_\alpha \times J$$
and $K_\alpha$ is compact, $\supp \lambda$ is closed.
Now if we choose $\alpha$ such that $K_\alpha$ is nonempty, every element of $K_\alpha$ uniquely determines a cocycle of labels, and hence a leaf of $\lambda$.
So $\supp \lambda$ is nonempty, and since all of its leaves are complete minimal, $\lambda$ is minimal.
\end{proof}

\subsubsection{Convergence in Thurston's geometric topology and $C^{1-}$}
At this stage of the argument we have constructed a limiting lamination with limiting flow boxes; we now check that the sequence of laminations actually converges to the limiting lamination.

If $K_\alpha$ is nonempty, then any $k_\alpha \in K_\alpha$ is the limit of some sequence $(k_\alpha^n)_n \in \prod_n K_\alpha^n$ \cite[Theorem 4.11]{nadler2017continuum}.
Thus $\{k_\alpha\} \times J$ can be written as the set of limits of sequences $(k_\alpha^n, x)_n \in \prod_n K_\alpha^n \times J$, and so any leaf $N$ of $\lambda$ takes the form $N = \lim_{n \to \infty} N_n$ for some sequence $(N_n) \in \prod_n \Leaves \lambda_n$, where $\Leaves \lambda_n$ is the set of leaves of $\lambda_n$.
In other words, leaves of $\lambda$ are pointwise limits of leaves in $\lambda_n$.

So it suffices to show that for $N \in \Leaves \lambda$, $P \in N$, and $P_n \to P$, where $P_n \in N_n$ and $N_n \in \Leaves \lambda_n$, $\normal_{N_n}(P_n) \to \normal_N(P)$.
To do this, suppose that $P \in U_\alpha$; $F_\alpha^n$ is close in tangential $C^\infty$ to $F_\alpha$, and the label $k^n_\alpha$ of $N_n$ is close to the label $k_\alpha$ of $N$.
In particular, if we consider $N$ and $N_n$ as graphs of functions $u, u_n$ in the coordinates induced by $F_\alpha$, then $u_n \to u$ in $C^\infty$; however, in such coordinates, $u$ is a constant.
A bootstrapping argument based on (\ref{nabla as a normal}) then shows that, since $\dif u_n \to 0$ in $C^0$, $\normal_{N_n} \to \partial_y = \normal_N$ in $C^0$ near $P$.
Therefore $\lambda_n \to \lambda$ in Thurston's geometric topology. Since we have already shown the convergence of the flow boxes in $C^{1-}$, we conclude convergence in the $C^{1-}$ flow box topology.

\subsubsection{Convergence in tangential $C^\infty$}
Next we check the convergence $F_\alpha^n \to F_\alpha$ in tangential $C^\infty$; this is slightly more subtle than convergence in $C^{1-}$ as the topology depends on the lamination.
Let $\nabla_{\rm tan} := (\partial_{y^1}, \dots, \partial_{y^{d - 1}})$ be the gradient tangent to $J$.
Since $F_\alpha^n$ is bounded in tangential $C^\infty$, for every $\ell \geq 1$ there exists $C_\ell > 0$ independent of $n$ such that if $k_\alpha^n \in K_\alpha^n$ and $y \in J$ then
\begin{equation}\label{bound in Cinfty}
|\nabla^\ell_{\rm tan} F_\alpha^n(k_\alpha^n, y)| \leq C_\ell.
\end{equation}
Let $k_\alpha^n \to k_\alpha$, so that $F_\alpha^n(k_\alpha^n, \cdot) \to F_\alpha(k_\alpha, \cdot)$ in $C^0$.
Observe that by the bound (\ref{bound in Cinfty}) with $\ell$ replaced by $\ell + 1$, we may diagonalize so that $\nabla^\ell_{\rm tan} F_\alpha^n(k_\alpha^n, \cdot) \to \nabla^\ell_{\rm tan} F_\alpha(k_\alpha, \cdot)$ in $C^0$.
So $F_\alpha^n \to F_\alpha$ in tangential $C^\infty$.
The inverse function theorem applied on each leaf of $\lambda$ implies that $(F_\alpha^n)^{-1} \to (F_\alpha)^{-1}$ in tangential $C^\infty$ as well.

\subsubsection{Maximality in Thurston's geometric topology}
Let $\lambda'$ be another lamination such that $\lambda_n \to \lambda'$ in the Thurston sense; we must show that $\lambda'$ is a sublamination of $\lambda$.
To this end, let $N$ be a leaf of $\lambda'$ and $\alpha \in A$ satisfy $N \cap U_\alpha \neq \emptyset$.
For every $x \in N \cap U_\alpha$ and $\varepsilon > 0$, there exist $n(x, \varepsilon)$, $N(x, \varepsilon) \in \Leaves \lambda_{n(x, \varepsilon)}$, and $y(x, \varepsilon)$ such that $\dist(x, y(x, \varepsilon)) < \varepsilon$ and $\dist_{SM}(\normal_N(x), \normal_{N(x, \varepsilon)}(y(x, \varepsilon))) < 2\varepsilon$.
As $\varepsilon \to 0$ we have $n(x, \varepsilon) \to \infty$, so if $N(x, \varepsilon) \cap U_\alpha = F_\alpha(\{k_\alpha^{n(x, \varepsilon)}\} \times J)$, as $\varepsilon \to 0$ along a subsequence we have $k_\alpha^{n(x, \varepsilon)} \to k_\alpha(x)$ for some $k_\alpha(x) \in K_\alpha$, which then corresponds to some $N(x) \in \Leaves \lambda$.
From the tangential $C^\infty$ convergence of $\lambda_n$ to $\lambda$, which in particular preserves all tangent vectors to $\lambda_n$ (hence preserves the normal bundle as well), we see that $\normal_{N(x, \varepsilon)}(y(x, \varepsilon)) \to \normal_{N(x)}(x)$.

So for every $x \in N$ there exists $N(x) \in \Leaves \lambda$ such that $x \in N(x)$ and $\normal_N(x) = \normal_{N(x)}(x)$.
In other words, for every $x \in N$, $N$ is tangent to $N(x)$ at $x$, and therefore we can represent $N$ as a graph
$$N \cap U_\alpha = (F_\alpha)_*(\{k = g(y): k \in K_\alpha, y \in J\})$$
where $\dif g(y) = 0$ for every $y \in J$.
We justify the use of flow box coordinates by the fact that $F_\alpha$ is tangentially $C^\infty$ smooth with respect to $\lambda$, hence preserves tangency to leaves of $\lambda$.
So $g$ is constant and hence $N = N(x)$ is actually a leaf of $\lambda$, as desired.

\subsubsection{Convergence in the measure topology}
Suppose that $\mu_n$ is transverse to $\lambda_n$.
Since $(T_{\mu_n})$ is vaguely bounded, along a subsequence, $T_{\mu_n} \to T$ for some current $T$ of locally finite mass.

Let $V$ be an open ball such that $\overline V$ is a compact subset of $M$ which does not intersect $\supp \lambda$.
By Lemma \ref{supports shrink in the limit} and the fact that $\lambda_n \to \lambda$, $V$ does not intersect $\supp \lambda_n$ for $n$ large enough, so that $\int_V \star |T_{\mu_n}| = 0$.
By (\ref{LSC}), $\int_V \star |T| = 0$, hence $\supp T \subseteq \supp \lambda$.

Since $(T_{\mu_n})$ is vaguely bounded, so is $(\mu_\alpha^n)$ for every $\alpha \in A_E$, so by taking a further subsequence, $\mu_\alpha^n \to \mu_\alpha$ vaguelyfor every $\alpha \in A_E$ and some positive Radon measures $\mu_\alpha$ (whose support is necessarily then contained in $K_\alpha$, and which is necessarily invariant under transition maps).
Taking the limit as $n \to \infty$ of the equation 
$$\int_{U_\alpha} T_{\mu_n} \wedge \varphi = \int_I \int_{\{k\} \times J} (F_\alpha^n)^* \varphi \dif \mu_\alpha^n(k),$$
and exploiting the fact that $F_\alpha^n \to F_\alpha$ in tangential $C^\infty$,
$$\int_{U_\alpha} T \wedge \varphi = \int_I \int_{\{k\} \times J} F_\alpha^* \varphi \dif \mu_\alpha(k).$$
In other words, $T$ is Ruelle-Sullivan for $(\lambda', \mu)$, where $\lambda'$ is the closure in $\lambda$ of $\supp \mu$.

It remains to show that $\lambda'$ is nonempty, or equivalently that $\mu$ is nonzero.
Since
$$\int_{\bigcup_{\alpha \in A_E} U_\alpha} \star |T_{\mu_n}| \geq \int_E \star |T_{\mu_n}| \geq \varepsilon$$
and $A_E$ is finite, a pigeonholing argument yields that, after taking a further subsequence, there exists $\varepsilon' > 0$ and $\alpha \in A_E$ such that $\int_{U_\alpha} \star |T_{\mu_n}| \geq \varepsilon'$.
By (\ref{transverse measure of an open set}) and the convergence of the flow boxes, it follows that for some $\varepsilon'' > 0$, $\mu_\alpha^n(I) \geq \varepsilon''$.
But $I$ is compact, so by the portmanteau theorem \cite[Theorem 13.16]{klenke2013probability}, $\mu_\alpha(I) \geq \varepsilon'' > 0$.

This completes the proof of Theorem \ref{compactness theorem}.

\subsection{Consequences of measured convergence}\label{relationships between modes}
We now apply Theorems \ref{main thm} and \ref{compactness theorem} to explain how the different modes of convergence are related to each other.
It is clear from the definitions that flow-box convergence implies Thurston convergence.
Moreover, for $d = 2$, Thurston claimed that that measure convergence implies Thurston convergence \cite[Proposition 8.10.3]{thurston1979geometry}, though he did not explicitly justify why the limit was geodesic, or why the convergence preserves the normal vectors.
We complete the proof that measure convergence implies Thurston convergence, and show that flow-box convergence sits in the middle of the chain of implications.

\begin{lemma}\label{limits of measured geodesic lams are geodesic}
Let $d \leq 7$, let $(\lambda_n, \mu_n)$ be measured minimal laminations in $M$ of bounded curvature, and $(\lambda_n, \mu_n) \to (\lambda, \mu)$.
Then $\lambda$ is minimal.
\end{lemma}
\begin{proof}
By Proposition \ref{minimal implies locally minimizing} and the curvature bound, for every $x \in M$ there exists $r > 0$ such that every leaf of every lamination $\lambda_n$ is area-minimizing in $B(x, r)$.
After shrinking $r$ if necessary, we may assume that $H^1(B(x, r), \RR) = 0$.
Then, by Theorem \ref{main thm}, the Ruelle-Sullivan currents $T_{\mu_n}$ on $B(x, r)$ are the exterior derivatives of functions $u_n$ of least gradient.
Since $u_n$ is only defined up to a constant, we impose $\int_M \star u_n = 0$, so by Poincar\'e's inequality,
$$\|u_n\|_{L^1(B)} \lesssim r\mu_n(B(x, r)) \lesssim r^d < \infty$$
for $n$ large.
In particular, $(u_n)$ is bounded in $BV$ and hence has a subsequence which converges in $L^1$ to a function $u$.
So by Proposition \ref{MirandaStability}, $u$ has least gradient and $\dif u_n \to \dif u$ vaguely.
Then $T_\mu = \dif u$, so the leaves of $\lambda|_{B(x, r)}$ are level sets of $u$.
By Theorem \ref{main thm of old paper}, the leaves of $\lambda|_{B(x, r)}$ are minimal; this is a local property, so $\lambda$ is minimal.
\end{proof}

\begin{proposition}\label{convergence of measures means Thurston convergence}
Let $d \leq 7$, let $(\lambda_n, \mu_n)$ be measured minimal laminations in $M$ of bounded curvature, and $(\lambda_n, \mu_n) \to (\lambda, \mu)$.
Then $\lambda_n \to \lambda$ in Thurston's geometric topology.
\end{proposition}
\begin{proof}
Let $x \in \supp \lambda$ and $\varepsilon > 0$ be small enough.
Then $\mu(B(x, \varepsilon)) > 0$, so by (\ref{LSC}), $\mu_n(B(x, \varepsilon)) > 0$ if $n$ is large enough.
Therefore, taking $x_n \in \supp \lambda_n \cap B(x, \varepsilon)$ and then taking $\varepsilon \to 0$, we find $x_n \to x$ with $x_n \in \supp \lambda_n$.
By Theorem \ref{regularity theorem}, the normal vectors $\normal_n$ of $\lambda_n$ extend to Lipschitz vector fields on all of $B(x, \varepsilon)$, where the Lipschitz constants are uniformly bounded.
So a subsequence converges to some vector field $\normal'$ in $C^0$.
By Lemma \ref{limits of measured geodesic lams are geodesic}, $\lambda$ is minimal, so by Theorem \ref{regularity theorem}, we again obtain a Lipschitz extension $\normal$ of the normal vector of $\lambda$.
By (\ref{polar ruelle sullivan}) and Lemma \ref{vague polar convergence}, $\normal = \normal'$, so $\normal_n(x_n) \to \normal(x)$.
\end{proof}

\begin{proposition}\label{convergence of traansverse measures means flow box convergence}
Let $d \leq 7$, $(\lambda_n, \mu_n)$ be measured minimal laminations in $M$ of bounded curvature, and $(\lambda_n, \mu_n) \to (\lambda, \mu)$.
Then $\lambda_n \to \lambda$ in the $C^{1-}$ and tangentially $C^\infty$ flow box topology, and in particular in the Thurston topology.
\end{proposition}
\begin{proof}
By Proposition \ref{convergence of measures means Thurston convergence}, $\lambda_n \to \lambda$ in Thurston's geometric topology.
After discarding some leaves of $\lambda_n$ we may assume that $\lambda$ is a maximal limit for the Thurston topology.
Moreover, every subsequence $(\lambda_{n_k})$ has a further subsequence $(\lambda_{n_{k_\ell}})$ which converges to some maximal limit $\tilde \lambda$ in the $C^{1-}$ flow box topology by Theorem \ref{compactness theorem}.
But convergence in the flow box topology implies convergence in Thurston's topology, so $\tilde \lambda = \lambda$.
Since $(\lambda_{n_k})$ was arbitrary, it follows that $\lambda_n \to \lambda$ in the $C^{1-}$ flow box topology.
\end{proof}

\subsection{Laminations in \texorpdfstring{$\Hyp^2$}{the hyperbolic plane}}\label{hyperbolic equivalence}
We now show that convergence of measured geodesic laminations in the hyperbolic plane $\Hyp^2$ as we have defined it is equivalent to convergence of measured geodesic laminations as defined by Thurston \cite[Chapter 8]{thurston1979geometry}.
Thus our results are compatible with the topology literature.

Recall that the space $G(\Hyp^2)$ of complete geodesics in $\Hyp^2$ with its $C^\infty$ topology is homeomorphic to the open M\"obius strip.
A geodesic lamination $\lambda$ is thus identified with a closed set $\hat \lambda \subset G(\Hyp^2)$ of nonintersecting geodesics \cite[\S8.5]{thurston1979geometry}.
By a \dfn{transverse measure} to $\lambda$, Thurston means a Radon measure $\hat \mu$ on $G(\Hyp^2)$ with support $\hat \lambda$.
In particular, he defines the convergence $(\lambda_n, \mu_n) \to (\lambda, \mu)$ to mean that $\hat \mu_n \to \hat \mu$ vaguely on $G(\Hyp^2)$ \cite[\S8.10]{thurston1979geometry}.
We show that a transverse measure $\mu$, in the sense of Definition \ref{transverse measure definition}, induces a transverse measure $\hat \mu$ in Thurston's sense.
Then we shall show that the two notions of convergence are also equivalent.

\begin{lemma}
Given $F \in C^0_\cpt(G(\Hyp^2))$ and a geodesic lamination $\lambda$ in $\Hyp^2$, there exists a $C^0_\cpt$ $1$-form $\varphi$ on $\Hyp^2$ such that for every $\gamma \in \hat \lambda$,
\begin{equation}\label{extend boundary function to interior}
F(\gamma) = \int_\gamma \varphi.
\end{equation}
Furthermore, $\varphi$ can be chosen to depend continuously in $C^0$ on $F$ in $C^0$ and $\lambda$ in the tangentially $C^1$ flow box topology.
\end{lemma}
\begin{proof}
Since $\hat K := \supp F$ is compact, there exists a compact set $K \Subset \Hyp^2$ such that every $\gamma \in \hat K$ intersects $K$.
We then let $B$ be an open ball containing $K$, and introduce a weight function $w \in C^0_\cpt(B, \RR_+)$ by first imposing that on a single leaf $\gamma \in \hat \lambda$, $\int_\gamma w|_\gamma \dif \mathcal H^1 = 1$ and $w|_\gamma = 0$ near $\partial B$.
By imposing that $w$ is constant along suitably chosen curves transverse to $\lambda$, we can impose that the same properties hold for every leaf $\gamma$.

We then define for every $\gamma \in \hat \lambda$ and $x \in \gamma$,
$$\varphi(x) := F(\gamma) w(x) \star \normal_\lambda^\flat(x).$$
Then $\varphi$ is continuous on $\supp \lambda$, since $\normal_\lambda$ is continuous by Theorem \ref{regularity theorem}.
We extend $\varphi$ by Urysohn's lemma to a $C^0_\cpt$ $1$-form on $\Hyp^2$.
By construction we have (\ref{extend boundary function to interior}) and continuous dependence on $F$.
The continuous dependence on $\lambda$ in the $C^1$ flow box topology follows, because this implies convergence of $\normal_\lambda$ in $C^0$, and implies local Hausdorff convergence of the leaves, which in turn implies convergence of the weight $w$.
\end{proof}

\begin{lemma}\label{existence of boundary measure}
Let $(\lambda, \mu)$ be a measured geodesic lamination in $\Hyp^2$.
Then there is a unique Radon measure $\hat \mu$ on $G(\Hyp^2)$ such that for every compactly supported $1$-form $\varphi$ on $\Hyp^2$,
\begin{equation}\label{Radon measure on boundary}
\int_{\Hyp^2} T_\mu \wedge \varphi = \int_{G(\Hyp^2)} \int_\gamma \varphi \dif \hat \mu(\gamma).
\end{equation}
Furthermore, $\supp \hat \mu = \hat \lambda$.
\end{lemma}
\begin{proof}
The condition (\ref{Radon measure on boundary}) is equivalent to
$$\int_{G(\Hyp^2)} F \dif \hat \mu = \int_{\Hyp^2} T_\mu \wedge \varphi,$$
where $\varphi$ is any $C^0_\cpt$ $1$-form satisfying (\ref{extend boundary function to interior}).
The choice of $\varphi$ is irrelevant: if we chose a different $1$-form $\psi$ satisfying (\ref{extend boundary function to interior}), then we would have $\int_\gamma \varphi = \int_\gamma \psi$ for every $\gamma \in \hat \lambda$, and then a computation using partitions of unity would show that $\int_{\Hyp^2} T_\mu \wedge \varphi = \int_{\Hyp^2} T_\mu \wedge \psi$.
The continuous dependence of $\varphi$ on $F$ implies that $F \mapsto \int F \dif \hat \mu$ is continuous, and it is clear that $F \mapsto \varphi$ can be chosen to be a linear map.
Therefore $\hat \mu$ is a well-defined Radon measure.
\end{proof}

\begin{proposition}
Let $(\lambda_n, \mu_n)$ be a sequence of measured geodesic laminations in $\Hyp^2$, and $(\lambda, \mu)$ a geodesic lamination.
Then $(\lambda_n, \mu_n) \to (\lambda, \mu)$ iff $\hat \mu_n \to \hat \mu$ vaguely on $G(\Hyp^2)$.
\end{proposition}
\begin{proof}
Let $F \in C^0_\cpt(G(\Hyp^2))$.
We first observe that for any measured geodesic lamination $(\kappa, \nu)$, if we choose $\psi$ to be a $1$-form extending $F$ with respect to $\kappa$ as in (\ref{extend boundary function to interior}), then 
\begin{equation}\label{convergence of boundary measures is convergence of ruelle-sullivan}
\int_{G(\Hyp^2)} F \dif \hat \nu = \int_{G(\Hyp^2)} \int_\gamma \psi \dif \hat \nu(\gamma) = \int_{\Hyp^2} T_\nu \wedge \psi.
\end{equation}
Choose $\varphi_n, \varphi$ to be $1$-forms extending $F$ with respect to $\lambda_n, \lambda$ as in (\ref{extend boundary function to interior}).
If $(\lambda_n, \mu_n) \to (\lambda, \mu)$, then $\lambda_n \to \lambda$ in the tangentially $C^1$ flow box topology by Proposition \ref{convergence of traansverse measures means flow box convergence}, so we may assume $\varphi_n \to \varphi$.
Using the convergences $\varphi_n \to \varphi$ and $T_{\mu_n} \to T_\mu$, and (\ref{convergence of boundary measures is convergence of ruelle-sullivan}), we compute 
\begin{align*}
\lim_{n \to \infty} \int_{G(\Hyp^2)} F \dif \hat \mu_n
&= \lim_{n \to \infty} \int_{\Hyp^2} T_{\mu_n} \wedge \varphi_n 
= \int_{\Hyp^2} T_\mu \wedge \varphi 
= \int_{G(\Hyp^2)} F \dif \hat \mu.
\end{align*}

Conversely, if $\hat \mu_n \to \hat \mu$, then we fix a $C^0_\cpt$ $1$-form $\varphi$ on $\Hyp^2$ and define $F(\gamma) := \int_\gamma \varphi$.
Then by (\ref{convergence of boundary measures is convergence of ruelle-sullivan}),
\begin{align*}
\lim_{n \to \infty} \int_{\Hyp^2} T_{\mu_n} \wedge \varphi
&= \lim_{n \to \infty} \int_{G(\Hyp^2)} F \dif \mu_n
= \int_{G(\Hyp^2)} F \dif \mu 
= \int_{\Hyp^2} T_\mu \wedge \varphi. \qedhere
\end{align*}
\end{proof}

\appendix 
\section{Geometric measure theory}\label{boundary appendix}
\subsection{Radon measures}\label{portmanteau appendix}
Let $X$ be a metrizable space, and let $C_\cpt(X)$ be the space of compactly supported continuous functions $f: X \to \RR$.
Its dual $C_\cpt(X)'$ is canonically isomorphic to the space of signed Radon measures on $X$, where the bilinear pairing is given by integration.
The weakstar topology on $C_\cpt(X)'$ is known as the \dfn{vague topology} \cite[Chapter III, \S1.9]{bourbaki2003integration}.
Unpacking the definitions, a sequence $(\mu_n)$ of signed Radon measures converges vaguely to $\mu$ iff for every $f \in C_\cpt(X)$,
$$\lim_{n \to \infty} \int_X f \dif \mu_n = \int_X f \dif \mu.$$

Now let $X = M$ be a manifold, and consider instead the space $C_\cpt(M, \Omega^{d - \ell})$ of compactly supported continuous $(d - \ell)$-forms.
An $\ell$-\dfn{current of locally finite mass} (which we will simply call an \dfn{$\ell$-current}) is an element of the dual space $C_\cpt(M, \Omega^{d - \ell})'$ \cite{simon1983GMT}.
Notice that our convention is the algebraic-geometric convention: an $\ell$-current generalizes an $\ell$-form or an $\ell$-dimensional submanifold.
This is not the geometric-measure-theoretic convention.

We denote the pairing of a $d - \ell$-current $T$ and an $\ell$-form $\varphi$ by $\int_M T \wedge \varphi$; this defines the vague topology on the space of currents.
Suppose that $M$ is Riemannian.
Then, to any $d - \ell$-current $T$ we may associate a positive Radon measure, its \dfn{mass measure} $\star |T|$, which satisfies for any function $f$,
$$\int_M f \star |T| := \sup_{|\varphi| \leq |f|} \int_M T \wedge \varphi,$$
and a $|T|$-measurable $d - \ell$-form $\psi$, the \dfn{polar part} \cite[Theorem 4.14]{simon1983GMT}, which satisfies $T = \psi |T|$, $|T|$-almost everywhere.

A set $\mathscr S$ of currents is \dfn{vaguely bounded} if for every precompact domain $U \Subset M$ there exists $C_U > 0$ such that for every $T \in \mathscr S$, $\int_U \star |T| \leq C_U$.
One can show using \cite[Chapter III, Proposition 15]{bourbaki2003integration} that $\mathscr S$ is vaguely bounded iff $\mathscr S$ is vaguely precompact.

The mapping $T \mapsto \int_M \star |T|$ is not vaguely continuous.
However, it is vaguely lower-semicontinuous: if $U$ is an open set and $T_n \to T$ vaguely on $M$, then by an easy generalization of \cite[Theorem 1.9]{Giusti77},
\begin{equation}\label{LSC}
\int_U \star |T| \leq \liminf_{n \to \infty} \int_U \star |T_n|.
\end{equation}
We also have a criterion for the polar part to be vaguely continuous:

\begin{lemma}\label{vague polar convergence}
Assume that $\psi_n, \psi, \psi'$ are continuous unit $d - \ell$-forms on $M$, such that $\psi_n \to \psi'$ in $C^0$.
Consider currents $T_n := \psi_n \mu_n$ and $T := \psi \mu$, where $\mu_n, \mu$ are positive Radon measures.
If $T_n \to T$ vaguely, and $x \in \supp T$, then $\psi(x) = \psi'(x)$.
\end{lemma}
\begin{proof}
Let $T_n' := \psi' \mu_n$.
Let $\varphi$ be a $C^0_\cpt$ $\ell$-form, and observe that since $(T_n)$ is vaguely precompact, there exists $C_\varphi > 0$ such that for every $n$, $\int_M |\varphi| \dif \mu_n \leq C_\varphi$.
For every $\varepsilon > 0$, and all sufficiently large $n$,
\begin{align*}
\left|\int_M (T_n' - T) \wedge \varphi\right|
&\leq \left|\int_M (T_n - T_n') \wedge \varphi\right| + \left|\int_M (T_n' - T) \wedge \varphi\right| \\
&\leq \|\psi' - \psi_n\|_{C^0} \int_M |\varphi| \dif \mu_n + \left|\int_M (T_n - T) \wedge \varphi\right| \\
&\leq (1 + C_\varphi) \varepsilon.
\end{align*}
Taking $\varepsilon \to 0$, we see that $T_n' \to T$ vaguely.
Testing $T_n'$ against $f \star \psi'$ where $f \in C^0_\cpt$ is arbitary, we conclude that $\mu_n \to \langle \psi, \psi'\rangle \mu$ vaguely.
But testing against $f \star \psi$, we have $\langle \psi, \psi'\rangle \mu_n \to \mu$ vaguely.
It follows that $\langle \psi, \psi'\rangle = 1$, and since both are unit forms, this is only possible if $\psi = \psi'$.
\end{proof}

\subsection{Boundaries}\label{boundary conventions}
If $X$ is a metric measure space and $E \subseteq X$ is a measurable set, then it is standard dogma in measure theory that $E$ should be viewed not as a set, but as the \emph{equivalence class} of sets up to symmetric difference with sets of measure zero.
Therefore the boundary $\partial^{\rm top} E := \overline E \cap \overline{M \setminus E}$ in the sense of point-set topology is ill-defined, unless $E$ has some canonical representative.
This issue is rectified by the following definition.

\begin{definition}
Let $(X, \mu)$ be a metric measure space and $E \subseteq X$ a measurable set.
The \dfn{measure-theoretic boundary}\footnote{There are multiple, nonequivalent definitions of the boundary operator on metric measure spaces \cite[Chapter 6]{Pugh02}.} $\partial^{\rm meas} E$ is the set of $x \in X$ such that for every sufficiently small $\varepsilon > 0$,
$$0 < \frac{\mu(E \cap B(x, \varepsilon))}{\mu(B(x, \varepsilon))} < 1.$$
\end{definition}

If $X = M$ is a Riemannian manifold with its Riemannian measure $\mu$, then there exists a representative $\tilde E$ of the equivalence class $E$ such that $\partial^{\rm top} \tilde E = \partial^{\rm meas} E$ \cite[Theorem 3.1]{Giusti77}; it is straightforward to check that $\partial^{\rm meas} E$ and $\tilde E$ do not depend on the choice of Riemannian metric.
\emph{We adopt the convention throughout this paper that we are working with representatives $\tilde E$ such that $\partial^{\rm top} \tilde E = \partial^{\rm meas} E$, and simply write $\partial E$ for the boundary.}
Since $\partial E = \partial^{\rm top} \tilde E$, $\partial E$ is a closed set. 

Let $E$ be a \dfn{set of locally finite perimeter}, meaning that $\dif 1_E$ has locally finite mass.
The polar part of $\dif 1_E$ is a $1$-form which we identify with the unit conormal $\normal_E^\flat$.
By an easy generalization of \cite[Theorem 4.4]{Giusti77},
\begin{equation}\label{hausdorff measure is jump mass}
(\mathcal H^{d - 1})|_{\partial E} = \star |\dif 1_E|.
\end{equation}




\subsection{Functions of bounded variation}\label{BV appendix}
A function $u$ has \dfn{locally bounded variation} (denoted $u \in BV_\loc(M)$) if its distributional derivative $\dif u$ is a $d - 1$-current of locally finite mass, and a measurable set $E \subseteq M$ has \dfn{locally finite perimeter} if $1_E \in BV_\loc(M)$.
If $\dif u$ has finite mass, thus $\int_M \star |\dif u| < \infty$, then $u$ has \dfn{bounded variation}, denoted $u \in BV(M)$.

Let $u \in BV_\loc(M)$.
Combining \cite[Theorem 1.23]{Giusti77}\footnote{The cited proof is for euclidean space, but goes through on Riemannian manifolds unchanged.
The reader who prefers a proof of (\ref{coarea formula}) from more general principles may note that (\ref{coarea formula}) holds if $M$ is a Poincar\'e space \cite[Proposition 4.2]{Miranda03}.
This implies the result when $M$ is a Riemannian manifold, since we may cover any Riemannian manifold by Poincar\'e spaces.} with (\ref{hausdorff measure is jump mass}), we have the \dfn{coarea formula}
\begin{equation}\label{coarea formula}
\int_M \star |\dif u| = \int_{-\infty}^\infty \mathcal H^{d - 1}(\partial \{u > y\}) \dif y.
\end{equation}
In particular, for almost every $y \in \RR$, $\{u > y\}$ is a set of locally finite perimeter.

\begin{proposition}
Let $u \in BV_\loc(M)$. Then 
\begin{equation}\label{level sets define support}
	\supp \dif u = \overline{\bigcup_{y \in \RR} \partial \{u > y\}}.
\end{equation}
\end{proposition}
\begin{proof}
First observe that $M \setminus \supp \dif u$ is the set of $x \in M$ such that for some open neighborhood $V \ni x$, $u|_V$ is almost constant.
Hence if $\partial \{u > y\} \cap V$ is nonempty, we can write $V = V_+ \sqcup V_-$, where both $V_\pm$ have positive measure, $u|_{V_-} \leq y$, and $u|_{V_+} > y$, which contradicts that $u|_V$ is almost constant.
Therefore $\partial \{u > y\} \subseteq \supp \dif u$.

Conversely, if $x \notin \overline{\bigcup_{y \in \RR} \partial \{u > y\}}$, then there exists an open neighborhood $V \ni x$ such that for every $y \in \RR$, $\partial \{u > y\} \cap V$ is empty.
Therefore either $V \subseteq \{u > y\}$ or $V \subseteq \{u \leq y\}$.
Let $y$ be the real number at which $u$ transitions from $\{u > y\}$ meeting $V$ to $\{u \leq y\}$ meeting $V$.
(Such a number must exist, or else $|u|_V| = +\infty$, violating that $u \in L^1_\loc$.)
Then $u|_V = y$ almost everywhere, so $u|_V$ is almost constant, hence $x \notin \supp \dif u$.
\end{proof}

\begin{lemma}\label{monotonicity dichotomy}
Let $I \subseteq \RR$ be an open interval.
Let $u \in BV_\loc(I)$, and let $\sigma: \supp \dif u \to \{\pm 1\}$ be the polar part of $\dif u$.
For every $x \in I$, one of the following holds:
\begin{enumerate}
\item There exists $\varepsilon > 0$ such that $u$ is monotone on $(x - \varepsilon, x + \varepsilon)$. \label{monotone interval}
\item For every $\varepsilon > 0$ there exist $y, z \in (x - \varepsilon, x + \varepsilon) \cap \supp \dif u$ of $\sigma$ such that $\sigma(y) = 1$ and $\sigma(z) = -1$, and $y, z$ are Lebesgue points of $\sigma$ with respect to the measure $|\dif u|$.
\end{enumerate}
\end{lemma}
\begin{proof}
Assume that (\ref{monotone interval}) is false, and let $\varepsilon > 0$.
Since $u$ is not monotone on $I' := (x - \varepsilon, x + \varepsilon)$, there exist $y_1, y_2, z_1, z_2 \in I'$ such that:
\begin{enumerate}
\item $y_1 < y_2$ and $u(y_1) < u(y_2)$.
\item $z_1 < z_2$ and $u(z_1) > u(z_2)$.
\end{enumerate}
Furthermore, $u$ is continuous away from a countable set, so after slightly moving the $y_i$ and $z_i$ we may assume that $u$ is continuous at $y_i$ and $z_i$.
Then
$$\int_{y_1}^{y_2} \dif u = u(y_2) - u(y_1) > 0,$$
where the integral is unambiguously defined since $u$ does not jump at $y_1$ and $y_2$.
So there exists a Lebesgue point $y \in (y_1, y_2) \cap \supp \dif u$ such that $\sigma(y) = 1$.
A similar argument works for $z$.
\end{proof}

\subsection{Minimal hypersurfaces}\label{minimal surfaces}
Let us establish conventions for minimal and area-minimizing hypersurfaces.
Let $U \subseteq M$ be an open set.
A smooth embedded oriented hypersurface $N \subset U$ such that $\partial N \subset \partial U$ is \dfn{area-minimizing} if, for every smooth embedded oriented hypersurface $L \subset U$ such that $\partial L = \partial N$,
$$\mathcal H^{d - 1}(N) \leq \mathcal H^{d - 1}(L).$$
We restrict to oriented hypersurfaces, because we are mainly interested in the case that $H^1(U, \RR) = 0$ and $N = \partial E$ for some open $E \subset U$, in which case $N$ is oriented by the inward unit normal.

Let $N = \partial E$ be a smooth hypersurface. By the coarea formula, $N$ is area-minimizing iff $1_E$ has least gradient. 
The regularity theorem for area-minimizers asserts that the assumption of smoothness is unnecessary here:

\begin{theorem}\label{regularity}
Let $d \leq 7$. Let $E$ be a set of locally finite perimeter in the Riemannian manifold $M$, such that $1_E$ has least gradient.
Then $\partial E$ is a smooth area-minimizing hypersurface.
\end{theorem}

This is well-known; see, for example, \cite{DeLellis18} or \cite[Chapter 8]{morgan2016geometric}.
Another proof is available for euclidean space which explicitly uses functions of least gradient \cite[Theorem 10.11]{Giusti77}; my expository note \cite{BackusFLG} explains how to extend that argument to Riemannian manifolds.

A hypersurface $N \subset U$ is \dfn{minimal} if the mean curvature $H_N := \tr \Two_N$ is $0$.
Thus every area-minimizing hypersurface is minimal, but not conversely.

\begin{lemma}
Let $N \subset M$ be a minimal hypersurface and $p \in N$. Then there exist $c, r_0 > 0$ such that for every $0 < r \leq r_0$,
\begin{equation}\label{monotonicity formula}
\mathcal H^{d - 1}(N \cap B(p, r)) \geq cr^{d - 1}.
\end{equation}
\end{lemma}
\begin{proof}
Let $A, r_0$ be as in the monotonicity formula \cite[Theorem 7.11]{MarquesXX}. Then the quantity 
$$f(r) := \frac{e^{Ar^2} \mathcal H^{d - 1}(N \cap B(p, r))}{\omega_{d - 1} r^{d - 1}}$$
is nondecreasing on $(0, r_0]$.
So the density $\theta := \lim_{r \to 0} f(r)$ of $N$ at $p$ exists, and
$$\mathcal H^{d - 1}(N \cap B(p, r)) \geq e^{-Ar^2} \theta r^{d - 1}.$$
Since $N$ is smooth, it has a tangent space at $p$, so $\theta \geq 1$.
The result follows with $c := \mathcal \theta e^{-Ar_0^2}$ (since we may assume $r_0 < \infty$).
\end{proof}

\begin{theorem}[{\cite[Theorem 7.2]{WhiteNotes}}]\label{compactness for minimal surfaces}
Let $(N_n)$ be a sequence of complete minimal hypersurfaces in an open set $U$.
Assume that for every $V \Subset U$ there exist $A, C > 0$ such that for every $n$, $\|\Two_{N_n}\|_{C^0(N_n \cap V)} \leq A$ and $\mathcal H^{d - 1}(N_n \cap V) \leq C$.
Then there is a complete minimal hypersurface $N \subset U$ such that $N_n \to N$ in $C^\infty$.
\end{theorem}

\section{Minimal hypersurfaces are locally area-minimizing} \label{locally minimizing appendix}
\begin{proposition}\label{minimal implies locally minimizing}
Let $d \leq 7$, let $i, A > 0$, and let $g$ be a Riemannian metric on $M$ with $\|\Riem_g\|_{C^1} < \infty$ and injectivity radius $\geq i$.
Then there exists $\delta_* > 0$ which only depends on $\|\Riem_g\|_{C^1}, i, A$, such that the following holds:

For every ball $B(p, \delta) \subset M$ with $\delta \leq \delta_*$ and every oriented minimal hypersurface $N \subset B(p, \delta)$ with $p \in N$, $\partial N \subset \partial B(p, \delta)$, $\|\Two_N\|_{C^0} \leq A$, and trivial normal bundle, $N$ is the unique area-minimizing hypersurface with boundary $\partial N$.
\end{proposition}

If we wanted to allow $\delta_*$ to depend on $N$ and not just $A$, then this proposition would follow from \cite[Theorem 2.1]{Lawlor96}.
In fact, one could likely modify Lawlor and Morgan's slicing argument to obtain the desired dependency.
Instead, we give a proof of Proposition \ref{minimal implies locally minimizing} which was originally sketched by Otis Chodosh in a MathOverflow post \cite{MathOverflowMinimalLocal}.
We take the results of \S\S\ref{Regularity}-\ref{level set estimates} as given in this appendix, as we shall not use Proposition \ref{minimal implies locally minimizing} in those sections.

Given a hypersurface $\Sigma \subset \Ball^d$ and a metric $h$ on $\Ball^d$, let $|\Sigma|_h := \mathcal H^{d - 1}(\Sigma)$ be the surface area (with respect to $h$), let
\begin{equation}\label{stability formula}
Q_{\Sigma, h} = -\Delta_{\Sigma, h} - |\Two_{\Sigma, h}|^2 - \Ric_h(\normal_{\Sigma, h}, \normal_{\Sigma, h})
\end{equation}
be the stability operator for $\Sigma$, and let $H_{\Sigma, h}$ be the mean curvature.
Then the second variation of area in the direction of a normal variation $s$ is \cite[Theorem 1.3]{Chodosh21}
\begin{equation}\label{second variation formula}
\delta^2_s |\Sigma|_h = \langle s, Q_{\Sigma, h} s\rangle_{L^2(\Sigma, h)} + \|sH_{\Sigma, h}\|_{L^2(\Sigma, h)}^2 + \int_\Sigma H_{\Sigma, h} s \star_{\Sigma, h} 1.
\end{equation}

\begin{lemma}\label{uniform continuity of stability}
Suppose that $s \in W^{1, 2}(\Ball^{d - 1})$ satisfy $s|_{\partial \Ball^{d - 1}} = 0$.
Let $\mathscr N$ be the set of graphs in $\Ball^d$ of functions lying in some compact subset of $C^2(\Ball^{d - 1})$ (so $s$ induces a normal variation on each member of $\mathscr N$).
Let $\mathscr G$ be a compact set of Riemannian metrics on $\Ball^d$ for the $C^2$ topology.
Then
\begin{align*}
\mathscr N \times \mathscr G &\to \RR, \\
(\Sigma, h) &\mapsto \delta^2_s |\Sigma|_h
\end{align*}
is uniformly continuous on $\mathscr N$, with modulus of continuity only depending on $\|s\|_{W^{1, 2}}, \mathscr N, \mathscr G$.
\end{lemma}
\begin{proof}
First observe that the mapping
\begin{equation}\label{stability curvature map}
(\Sigma, h) \mapsto (Q_{\Sigma, h}, H_{\Sigma, h})
\end{equation}
is continuous, where $Q_{\Sigma, h}$ is measured using the $C^0$ topology on its coefficients, and $H_{\Sigma, h}$ is measured in $C^0$.
By (\ref{second variation formula}), it follows that $\delta^2_s |\Sigma|_h$ is a continuous function of $(\Sigma, h)$, and its modulus of continuity at $(\Sigma, h)$ only depends on $\|s\|_{W^{1, 2}}$.
The result then follows from the compactness of $\mathscr N \times \mathscr G$.
\end{proof}

\begin{proof}[Proof of Proposition \ref{minimal implies locally minimizing}]
We proceed by compactness and contradiction.
By rescaling, we may assume that $A = 1/100$.
Let $(g_j)$ be a sequence of metrics on $\Ball^d$ with $\|\Riem_g\|_{C^1} \lesssim 1$, which are increasingly severe violations of the proposition, in the sense that there exist oriented $g_j$-minimal hypersurfaces $N_j \subset B_j := B_{g_j}(0, 1/j)$ which contain $0$ and are not uniquely area-minimizing, but such that $\|\Two_{N_j}\|_{C^0(B_j)} \leq 1/100$.

We may assume that the $(g_j)$ are normal coordinates centered on $0$.
In particular,
$$\|(g_j)_{\mu \nu} - \delta_{\mu \nu}\|_{C^3} \lesssim \|\Riem_{g_j}\|_{C^1},$$
implying that $(g_j)$ is contained in a compact subset of $C^2$.
Furthermore, by rotating $N_j$ and applying the assumptions that $0 \in N_j$ and $\|\Two_{N_j}\|_{C^0(B_j)} \leq 1/100$ (and using the normal coordinate condition and Gauss-Codazzi to propagate this estimate to an estimate on the second fundamental form in the euclidean metric), we may assume (using Lemma \ref{existence of tubes}) that $N_j$ is the graph of a function, and is contained in a small tubular neighborhood of the equatorial disk in $B_j$.
Moreover, such a tubular neighborhood shrinks to the equatorial disk as $j \to \infty$.

Let $L_j \subset \Ball^d$ be the rescaling of $N_j$ by $j$, let $g_j'$ be the corresponding rescaling of $g_j$, and let $f_j$ be the function on (a large subset of) $\Ball^{d - 1}$ whose graph is $L_j$.
By elliptic estimates as in \S\ref{Leaf estimates}, and the estimates on $N_j$ described in the previous paragraph, $f_j \to 0$ in $C^3$.

Let $L_j'$ be an area-minimizing competitor of $L_j$ with respect to $g_j'$, which exists by a modification of \cite[Theorem 1.20]{Giusti77} and Theorem \ref{main thm of old paper}.
So we have bounds on $|L_j'|_{g_j'}$ (by Lemma \ref{least gradient area bound}) and $\|\Two_{L_j', g_j'}\|_{C^0}$ (by Lemma \ref{least gradient curvature bound}).
By (a standard generalization of) Theorem \ref{compactness for minimal surfaces}, $L_j'$ is the graph of a function $f_j'$ which converges to $0$ in $C^3$.
In particular, $L_j, L_j'$ are graphs over each other.
Viewing $L_j'$ as a normal variation of $L_j$, let $s_j$ be the function on $L_j$ defining the normal variation $L_j'$.
Then $s_j \to 0$ in $W^{1, 2}$.
Let $Q_j$ be the stability operator $Q_{L_j, g_j'}$, and observe that since $L_j, L_j'$ are converging to the equatorial disk $\Ball^{d - 1}$ in $C^2$, they lie in a compact subset of $C^2$. 
Moreover, $L_j$ has zero $g_j'$-mean curvature, so the $g_j'$-area of $L_j'$ satisfies
\begin{equation}\label{stability area bound}
|L_j'|_{g_j'} \geq |L_j|_{g_j'} + \langle s_j, Q_j s_j\rangle_{L^2(L_j, g_j')} + o(\|s_j\|_{W^{1, 2}(L_j, g_j')}^2)
\end{equation}
where the rate of convergence in the error term is independent of $j$ by Lemma \ref{uniform continuity of stability}.
By (\ref{norms on uk}), $\|s_j\|_{W^{1, 2}} \sim \|s_j\|_{L^2}$, which we use to replace the error term in (\ref{stability area bound}).

Let $\lambda > 0$ be the first Dirichlet eigenvalue of the Laplacian $-\Delta$ of $\Ball^{d - 1}$, and $\lambda_j$ the first Dirichlet eigenvalue of $Q_j$.
Observe that by (\ref{stability formula}) and the convergence $\|f_j\|_{C^3} + \|g_j'\|_{C^3} \to 0$, we can decompose 
\begin{align}
Q_j &= -\Delta + R_j, \label{perturbing Laplacian 1}\\
R_j &:= \partial_\mu a_j^{\mu \nu} \partial_\nu + b_j^\mu \partial_\mu + c_j, \label{perturbing Laplacian 2}
\end{align}
where, as $j \to \infty$,
\begin{equation}
\|a_j\|_{C^1} + \|b_j\|_{C^0} + \|c_j\|_{C^0} \to 0. \label{perturbing Laplacian 3}
\end{equation}
Let $u_j$ be an eigenfunction corresponding to $\lambda_j$ such that $\|u_j\|_{L^2} = 1$.
Using (\ref{perturbing Laplacian 1}--\ref{perturbing Laplacian 3}), and the elliptic estimate \cite[Theorem 8.12]{gilbarg2015elliptic}, we see that for $j$ large enough, 
$$\|u_j\|_{W^{2, 2}} \lesssim \|u_j\|_{L^2} + \|\lambda_j u_j\|_{L^2} \leq 1 + |\lambda_j|.$$
Moreover, $|\langle u_j, R_j u_j\rangle| = o(\|u_j\|_{W^{2, 2}})$, so by the Rayleigh-Ritz theorem,
\begin{align*}
\lambda_j 
&= \langle u_j, Q_j u_j\rangle 
= \langle u_j, -\Delta u_j\rangle + \langle u_j, R_j u_j\rangle 
\geq \lambda - o(1 + |\lambda_j|).
\end{align*}
Rearranging terms, we see that for $j$ large enough, $\lambda_j \geq \lambda/2$.
Since $L_j'$ is area-minimizing, we obtain from (\ref{stability area bound}) that for $j$ large enough,
$$|L_j'|_{g_j'} \geq |L_j|_{g_j'} + \frac{\lambda}{2} \|s_j\|_{L^2}^2 + o(\|s_j\|_{L^2}^2),$$
but since $\lambda > 0$ and $|L_j'|_{g_j'} \leq |L_j|_{g_j'}$, we conclude that if $j$ is large, then $s_j = 0$.
Therefore $L_j = L_j'$, even though $L_j$ is not uniquely area-minimizing.
\end{proof}

\printbibliography

\end{document}